\documentclass[10pt]{amsart}
\usepackage{amssymb,amsmath,amsthm,bbm,amsfonts,amsbsy}
\usepackage{cancel}
\usepackage[shortlabels]{enumitem}
\numberwithin{equation}{subsection}
\newtheorem{thm}{Theorem}[section]
\newtheorem*{thm*}{Theorem}

\newtheorem{lem}[thm]{Lemma}
\theoremstyle{definition}
\newtheorem{definition}[thm]{Definition}
\newtheorem{rem}[thm]{Remark}
\theoremstyle{remark}

\allowdisplaybreaks[1]
\newcommand{\R}{\mathbb{R}}
\newcommand{\C}{\mathbb{C}}
\newcommand{\V}{\mathbb{V}}
\newcommand{\HH}{\mathbb{H}}
\newcommand{\PP}{\mathbb{P}}
\newcommand{\N}{\mathcal{N}}
\newcommand{\F}{\mathcal{F}}
\newcommand{\B}{\mathcal{B}}
\newcommand{\LK}{\mathcal{K}}
\newcommand{\im}{\mathbf{i}}
\newcommand{\dd}{\mathrm{d}}
\newcommand{\K}{\mathtt{K}}
\newcommand{\U}{\mathcal{U}}

\begin{document}

\title{Unit Tangent Bundles, CR Leaf Spaces, and Hypercomplex Structures}
\author{Curtis Porter} 
\address{Department of Mathematics, Duke University, 120 Science Dr. Durham, NC 27710}
\email{cwp19@math.duke.edu}

\begin{abstract}
Unit tangent bundles $UM$ of semi-Riemannian manifolds $M$ are shown to be examples of dynamical Legendrian contact structures, which were defined in recent work \cite{SykesZelenko} of Sykes-Zelenko to generalize leaf spaces of 2-nondegenerate CR manifolds. In doing so, Sykes-Zelenko extended the classification in Porter-Zelenko \cite{PZ} of regular, 2-nondegenerate CR structures to those that can be recovered from their leaf space. The present paper treats dynamical Legendrian contact structures associated with 2-nondegenerate CR structures which were called \emph{strongly regular} in Porter-Zelenko, named \emph{L-contact structures}. Closely related to Lie-contact structures, L-contact manifolds have homogeneous models given by isotropic Grassmannians of complex 2-planes whose algebra of infinitesimal symmetries is one of $\mathfrak{so}(p+2,q+2)$ or $\mathfrak{so}^*(2p+4)$ for $p\geq1$, $q\geq0$. Each 2-plane in the homogeneous model is a split-quaternionic or quaternionic line, respectively, and more general L-contact structures arise on contact manifolds with hypercomplex structures, unit tangent bundles being a prime example. The Ricci curvature tensor of $M$ is used to define the \emph{Ricci-shifted} L-contact structure on $UM$, whose Nijenhuis tensor vanishes when $M$ is conformally flat. In the language of Sykes-Zelenko (for $M$ analytic), such $UM$ is the leaf space of a 2-nondegenerate CR manifold which is \emph{recoverable} from $UM$, providing a new source of examples of 2-nondegenerate CR structures.
\end{abstract}

\maketitle
\tableofcontents


\section{Introduction}


A CR manifold $\U$ whose Levi form $\mathcal{L}$ has $k$-dimensional kernel is foliated by complex submanifolds of complex dimension $k$, hence has a local product structure $\U\approx U\times\C^k$, where $U$ is the leaf space of the Levi foliation. For concreteness, we can take $\U$ to be a regular level set of a smooth, non-constant function $\rho:\C^{n+k+1}\to\R$ so that $T\U=\ker\dd\rho$ contains a corank-1 distribution given by the kernel $D=\ker\partial\rho\subset T\U$ of the holomorphic differential. The CR structure of $\U$ is the splitting $\C D=H\oplus\overline{H}$ of the complexification of $D$, where $H,\overline{H}$ are the intersections with $\C T\U$ of the holomorphic and anti-holomorphic bundles of the ambient space $\C^{n+k+1}$. The Levi kernel $\LK\subset H$ consists of null directions for $\mathcal{L}=\partial\overline{\partial}\rho$, so that $\LK\oplus\overline{\LK}$ is the complexified tangent bundle of the leaves of the Levi foliation (we always assume $\LK$ has constant rank). 

The Levi foliation of $\U$ introduces an equivalence relation on $\U$ -- two points being equivalent if they lie in the same leaf -- and the quotient map $Q:\U\to U$ sends each point to its equivalence class in the $(2n+1)$-dimensional leaf space (since our considerations are local, we can assume without further comment that the leaf space is a manifold and the quotient map is smooth). A local trivialization $\U\approx U\times\C^k$ adapted to the Levi foliation is a \emph{straightening} (\cite[\S5.2]{ChirkaCR}, \cite{FreemanStraight}) if it preserves the CR structure of $\U$; equivalently, $\U$ is \emph{straightenable} if $Q_*H\subset\C TU$ is a well-defined CR structure on the leaf space. In this article we consider the opposite extreme, where $\U$ is \emph{2-nondegenerate} (\S\ref{2nonsec}), so that $U$ does not inherit a canonical CR structure from $\U$, but rather a $k$-(complex)-dimensional family 
\begin{align}\label{QHinU}
\{Q_*H_{\widetilde{u}}\subset\C T_uU\ |\ \widetilde{u}\in Q(\widetilde{u})=u\}
\end{align}
of \emph{almost}-CR structures at each $u\in U$. In this sense, a 2-nondegenerate CR manifold determines a highly nontrivial fibration over its leaf space. 

Classification of 2-nondegenerate CR manifolds has been an active research program for the past decade, motivating new developments in the method of equivalence (see \cite{PZ}, \cite{SykesZelenko}, and references therein). In general, the method of equivalence classifies manifolds carrying some geometric structure by realizing them as curved versions of a homogeneous model in the spirit of E. Cartan's \emph{espaces g\'{e}n\'{e}raliz\'{e}s} \cite[Preface]{Sharpe}. Successful application of the method to 2-nondegenerate CR manifolds of arbitrary (odd) dimension in \cite{PZ} required a \emph{regularity} assumption on the fibers of $Q:\U\to U$. Regularity is a stringent requirement, as it is generically absent in the moduli space of all possible 2-nondegenerate CR structures, but recent work \cite{SykesZelenko} of Sykes-Zelenko generalizing beyond the regular setting showed that (in the pseudo-convex case, and for arbitrary signature in dimensions 7 and 9) it is also generically true that non-regular symbols do not admit homogeneous models.   

A key strategy of Sykes-Zelenko is to shift focus from the 2-nondegenerate CR manifold $\U$ to its leaf space $U$, as follows. $U$ has a contact distribution $\Delta=Q_*D$, and $\mathcal{L}$ descends to a symplectic form $\mathfrak{L}:\C\Delta\times\C\Delta\to\C(TU/\Delta)$ so that \eqref{QHinU} and its complex conjugate $Q_*\overline{H}$ define two $k$-dimensional complex submanifolds of the Lagrangian Grassmannian bundle $\mathsf{LaGr}(\C\Delta)\to U$ of $\mathfrak{L}$-Lagrangian (i.e., \emph{almost-CR}) splittings of $\C\Delta$. A \emph{dynamical Legendrian contact structure} (\cite[Def. 2.3]{SykesZelenko}) on a contact manifold $(U,\Delta)$ is exactly that: a pair of submanifolds of $\mathsf{LaGr}(\C\Delta)$ that are appropriately related by complex conjugation. Working in the analytic category\footnote{The geometric PDE underlying dynamical Legendrian contact structures is determined at finite jet order, so analyticity is only employed in \cite{SykesZelenko} to analytically continue local coordinate charts on $U$, because the complex-conjugation map on $\C U$ simplifies several constructions on bundles over $U$.}, Sykes-Zelenko pass to the complexification of $U$ in order to show that a 2-nondegenerate CR structure is \emph{recoverable} from the dynamical Legendrian contact structure on its leaf space, at least under certain hypotheses (\cite[Prop. 2.6]{SykesZelenko}) on Lie derivatives along $\LK$ which will always be satisfied for our purposes. Though we work in the smooth (real) category, and we are only interested in \emph{regular} 2-nondegenerate CR structures, this article aims to illustrate the scope of dynamical Legendrian contact geometry by presenting a large family of examples: unit tangent bundles of semi-Riemannian manifolds. 

Tangent sphere bundles of Riemannian manifolds are rich with structure. They have been studied extensively as examples of Riemannian manifolds in their own right \cite{UTBsurvey} (or, extrinsically, as Riemannian hypersurfaces \cite{UTextrinsic}) and in their capacity as contact-metric \cite{UTcontact}, Lie contact \cite{LC}, and CR manifolds \cite{Tanno, BarDrag}. For a semi-Riemannian manifold $(M,g)$ (where $g$ has at least one positive eigenvalue), the unit tangent bundle of $M$ is 
\begin{align}\label{UMdef}
UM=\{u\in TM\ |\ g(u,u)=1\}.
\end{align}
The Levi-Civita connection of $g$ determines a splitting $TTM=\overrightarrow{TM}\oplus T^\uparrow M$ into horizontal and vertical bundles canonically isomorphic to $TM$, and $g$ is lifted to the Sasaki metric (\cite{Sasaki}) $\hat{g}$ on $TM$ evaluating on $\overrightarrow{TM}$ and its orthogonal complement $T^\uparrow M$ exactly as $g$ evaluates on $TM$. Just as one identifies $T_u S^n\cong u^\bot\subset\R^{n+1}$, the vertical part of $T_u(U_xM)$ is the vertical lift of $u^\bot=\ker g(u,\cdot)\subset T_xM$. On all of $T_u(UM)$, $\theta|_u=\hat{g}(\overrightarrow{u},\cdot)\in\Omega^1(UM)$ is a contact 1-form, where $\overrightarrow{u}\in\overrightarrow{TM}$ is the horizontal lift of $u\in UM$. Thus, we have a contact distribution $\Delta\subset TUM$ and symplectic form $\mathfrak{L}:\C\Delta\times\C\Delta\to\C$,
\begin{align*}
&\Delta=\ker\theta,
&\mathfrak{L}=-\im\dd\theta.
\end{align*}

It remains to appoint a submanifold of $\mathsf{LaGr}(\C\Delta_u)$ for each $u\in UM$, and to this end we return to 2-nondegenerate CR structures for inspiration. Among regular structures classified in \cite{PZ}, \emph{strongly regular} 2-nondegenerate manifolds (with $k=\mathrm{rank}_\C\LK=1$) are modeled on homogeneous spaces whose Lie algebra $\mathfrak{g}$ of infinitesimal symmetries is a real form of $\mathfrak{so}(n+4,\C)$. Specifically, if $\mathcal{L}$ has signature $(p,q)$ for $p+q=n$, $\mathfrak{g}$ is one of $\mathfrak{so}(p+2,q+2)$ or $\mathfrak{so}^*(2p+4)$ (the latter is only possible if $q=p$). An essential observation for the present work is that the dynamical Legendrian contact structure on the leaf space $U$ of a strongly regular 2-nondegenerate CR structure $\U$ is generated by a \emph{hyper-CR} structure (\S\ref{AHCsec}) on $U$; i.e., a pair of endomorphisms $J,K:\Delta\to\Delta$ satisfying $J^2=-\mathbbm{1}$, $J\circ K=-K\circ J$, and $K^2=\varepsilon\mathbbm{1}$ (where $\mathbbm{1}$ is the identity and $\varepsilon=\pm1$). This is the ``CR version" of a hypercomplex structure, and the case $\varepsilon=1$ is called \emph{split-quaternionic} while $\varepsilon=-1$ (which requires $\tfrac{1}{2}\mathrm{rank}\Delta$ to be even) is \emph{quaternionic}.  

Our primary object of study is therefore a contact manifold $(U,\Delta)$ with a dynamical Legendrian contact structure given by a complex curve in the bundle $\mathsf{LaGr}(\C\Delta)$ which is generated by a hyper-CR structure on $J,K:\Delta\to\Delta$; in short, we call $U$ an \emph{L-contact\footnote{Here ``L" could plausibly refer to Lagrange, Legendre, Levi, Lie,  or just \emph{leaf}, so the reader can choose their favorite.} manifold}. For $U=UM$, where the signature of the semi-Riemannian metric $g$ is $(p+1,q)$, the \emph{standard} L-contact structure is generated by the split-quaternionic structure 
\begin{align*}
&J,K:TTM\to TTM,
&J=\begin{bmatrix}0&-\mathbbm{1}\\\mathbbm{1}&0\end{bmatrix}, 
&&K=\begin{bmatrix}0&\mathbbm{1}\\\mathbbm{1}&0\end{bmatrix}
&&\text{on}
&&TTM=\begin{array}{c}\overrightarrow{TM}\\\oplus \\T^\uparrow M\end{array}.
\end{align*}
An alternative L-contact structure is given by ``shifting" the standard one by the Ricci curvature tensor of $M$ (Definition \ref{RicSQ} in \S\ref{UTBsec}), which leads to the article's main result,

\begin{thm*}(\ref{mainthm}, \S\ref{ONFBsec})
In order for the Ricci-shifted L-contact structure of $UM$ to be the leaf space of a 2-nondegenerate CR manifold, it is sufficient that $M$ is conformally flat and real-analytic. 
\end{thm*}

\noindent L-contact structures of the quaternionic variety are available on $UM$ when the signature of $g$ is $(p+1,p)$. These are discussed in \S\ref{QLUM}. 

Homogeneous models of strongly regular 2-nondegenerate CR manifolds and their leaf spaces are presented in \S\ref{homogmodsec}, emphasizing the algebraic consequences of a (split) quaternionic structure. This offers a casual encounter with 2-nondegenerate CR geometry before the formal definitions of \S\ref{Lconsec}. Leaf spaces of 2-nondegenerate CR manifolds motivate the definition of L-contact manifolds in \S\ref{2nonsec}. Once defined, the initial steps of Cartan's method of equivalence indicate how L-contact structures may be classified in \S\S\ref{coframingssec}-\ref{QSEsec}. The constructions are similar in flavor to the author's thesis work (\cite{CPdis}, see also \cite{CPCAG} and its references), which may be more accessible than \cite{PZ}, but the conclusions of \S\ref{SQSEsec} and \S\ref{QSEsec} fundamentally rely on those of \cite{PZ} and \cite{SykesZelenko}. The L-contact structure of the unit tangent bundle is the subject of \S\ref{UMLcon}. We show in \S\ref{ONFBsec} that the orthonormal frame bundle of $(M,g)$ fibers over $UM$, realizing the structure equations of \S\S\ref{coframingssec}-\ref{QSEsec} for $U=UM$. In particular, this is sufficient to compare the L-contact structure of $UM$ to the homogeneous models in \S\ref{homogmodsec}. Finally, Appendix \ref{appendixsec} exhibits local hypersurface realizations of the 2-nondegenerate CR and L-contact models of \S\ref{SQlinessec} in order to explicitly realize the leaf space as a unit tangent bundle in \S\ref{FTsec}. \S\ref{Lieconsec} is a brief account of some similarities between L-contact and Lie contact geometry, which was a principal motivation for this work.

\vspace{\baselineskip}


\section{Homogeneous Models}\label{homogmodsec}


\subsection{Linear Algebra of (Split) Quaternionic Structures}\label{splitquatsec}


Denote $\im=\sqrt{-1}$, and recall the Pauli matrices, 
\begin{align}\label{Paulimat}
&\sigma_0=\begin{bmatrix}1&0\\0&1\end{bmatrix},
&\sigma_1=\begin{bmatrix}0&1\\1&0\end{bmatrix},
&&\sigma_2=\begin{bmatrix}0&-\im\\\im&0\end{bmatrix},
&&\sigma_3=\begin{bmatrix}1&0\\0&-1\end{bmatrix}.
\end{align}
The $\R$-algebra $\mathfrak{A}$ will refer to one of  
\begin{equation}\label{AHH}
\begin{aligned}
&\text{Quaternions}: &&\mathfrak{A}=\HH=\R\{\sigma_0,\im\sigma_1,-\im\sigma_2,\im\sigma_3\} 
&&=\left\{\left.\begin{bmatrix}w&-\overline{z}\\z&\overline{w}\end{bmatrix}\right| w,z\in\C\right\};\\
&\text{Split-Quaternions}: &&\mathfrak{A}=\cancel{\HH}=\R\{\sigma_0,\sigma_1,\sigma_2,\im\sigma_3\} 
&&=\left\{\left.\begin{bmatrix}w&\overline{z}\\z&\overline{w}\end{bmatrix}\hspace{1.9mm} \right| w,z\in\C\right\}.
\end{aligned} 
\end{equation}

Let $V$ be a real vector space and $\V=V\otimes\C$ its complexification, so that an $\R$-basis of $V$ is a $\C$-basis of $\V$. An \emph{$\varepsilon$-quaternionic structure} on $\V$ is given by a linear map $\K:V\to V$ with $\K^2=\varepsilon\mathbbm{1}$ (where $\mathbbm{1}$ is the identity and $\varepsilon=\pm1$) which is extended to $\V$ by conjugate-linearity: $\K(cv)=\overline{c}\K v$ for $c\in\C$ and $v\in V$. The case $\varepsilon=1$ is called \emph{split-quaternionic} and the case $\varepsilon=-1$ is called \emph{quaternionic}. Quaternionic structures may exist only when $\dim_\R V=\dim_\C\V$ is even. Write the external direct sum $\V\oplus\V$ as length-2 row vectors with $\V$-entries so that the right $\mathfrak{A}$-module structure given by matrix multiplication (on the right) is well-defined on the real subspace
\begin{align*}
\mathcal{V}=\left\{\begin{bmatrix}v&\K v\end{bmatrix}\ |\ v\in\V\right\}\cong\V.
\end{align*}  
The latter is an isomorphism of real vector spaces, by which we can say that the $\mathfrak{A}$-span of $v\in\V$ is 
\begin{align}\label{Aspanv}
v\mathfrak{A}=\left\{\left.\begin{bmatrix}v&\K v\end{bmatrix}\begin{bmatrix}w&\varepsilon\overline{z}\\z&\overline{w}\end{bmatrix}\ \right|\ w,z\in\C\right\}.
\end{align}
If $v\wedge\K v\neq0$, $v\mathfrak{A}$ is the complex 2-plane $\C\{v,\K v\}\in\mathsf{Gr}_2(\V)$ in the Grassmannian of complex 2-planes in $\V$. Equivalently, $v\wedge\K v\neq0\Rightarrow v\mathfrak{A}$ is the complex-projective line $\PP(v+c\K v)\subset\PP\V$ ($c\in\C$) through $\PP(v)$ in the complex-projectivization of $\V$. We also call \eqref{Aspanv} a split-quaternionic line $(\varepsilon=1)$ or quaternionic line $(\varepsilon=-1)$.

The action of $A\in GL(\V)$ on $\V\oplus\V$ is given by the standard action on the first summand and the conjugate $\overline{A}$ on the second, where conjugation on $GL(\V)$ is determined by conjugation on $\V$ with respect to the real subspace $V\subset\V$. Evidently, $A(\mathcal{V})\subset\mathcal{V}$ when $\overline{A}\K=\K A$. $G(\V,\K)\subset GL(\V)$ denotes the corresponding subgroup of symmetries of $\mathcal{V}$.

Let $V$ be equipped with a symmetric, nondegenerate bilinear form $b\in S^2V^*$ with respect to which $\mathtt{K}$ is symmetric,
\begin{align*}
&b(\K v_1,v_2)=b(v_1,\K v_2),
&v_1,v_2\in V.
\end{align*}
Note that $\K$ is $b$-orthogonal when $\varepsilon=1$, and when $\varepsilon=-1$ it must be that $b$ has split signature. $G(V,b)\subset GL(V)$ is the real orthogonal group specified by the signature of $b$. The $\C$-bilinear extension of $b$ to $\V$ is denoted $\boldsymbol{b}$, so that the complex orthogonal group $G(\V,\boldsymbol{b})$ is the complexification of $G(V,b)$. Antilinearity of $\K$ implies 
\begin{align}\label{CbK}
&\boldsymbol{b}(\K v_1,v_2)=\overline{\boldsymbol{b}(v_1,\K v_2)},
&v_1,v_2\in\V,
\end{align}
which defines a Hermitian form $\boldsymbol{h}(v_1,v_2)$ on $\V$. Moreover, \eqref{CbK} shows that $b$ extends naturally to an $\mathfrak{A}$-valued form on $\mathcal{V}$,
\begin{align*}
b_\mathfrak{A}(\begin{bmatrix}v_1&\K v_1\end{bmatrix},\begin{bmatrix}v_2&\K v_2\end{bmatrix})=
\sqrt{\varepsilon}\begin{bmatrix}\boldsymbol{b}(v_1,v_2)&\boldsymbol{b}(v_1,\K v_2)\\\boldsymbol{b}(\K v_1,v_2)&\boldsymbol{b}(\K v_1,\K v_2)\end{bmatrix}=
\begin{bmatrix}\sqrt{\varepsilon}\boldsymbol{b}(v_1,v_2)&\sqrt{\varepsilon}\overline{\boldsymbol{h}(v_1,v_2)}\\\sqrt{\varepsilon}\boldsymbol{h}(v_1,v_2)&\varepsilon\sqrt{\varepsilon}\overline{\boldsymbol{b}(v_1,v_2)}\end{bmatrix}.
\end{align*}
Thus, the symmetries $G(\mathcal{V},b_\mathfrak{A})$ of the pair $(\boldsymbol{b},\K)$ on $\V$ may be thought of as the intersection of complex-orthogonal $G(\V,\boldsymbol{b})$ and unitary symmetries $G(\V,\boldsymbol{h})$ of $\V$. Writing $\mathfrak{g}$ for the Lie algebra of the symmetry group $G(\mathcal{V},b_\mathfrak{A})$, we have in the quaternionic case $\mathfrak{g}=\mathfrak{so}^*(\dim V)$ and in the split-quaternionic case $\mathfrak{g}=\mathfrak{so}(\mathrm{sig}(b))$ where $\mathrm{sig}(b)$ is the signature of $b$ on $V$.

\vspace{\baselineskip}

\subsection{Split-Quaternionic Lines as Leaves of a 2-Nondegenerate Levi Foliation}\label{SQlinessec}

For $n\geq 1$, let $V=\R^{n+4}$ with $b\in S^2V^*$ and split-quaternionic $\K:V\to V$ represented in the standard basis of column vectors by
\begin{align*}
&b=\begin{bmatrix}0&0&\sigma_1\\0&\epsilon_{ij}&0\\\sigma_1&0&0\end{bmatrix},
&\K=\begin{bmatrix}\sigma_1&0&0\\0&\mathbbm{1}_n&0\\0&0&\sigma_1\end{bmatrix},
\end{align*} 
where 
\begin{align}\label{epsilonmetric}
&\epsilon_{ij}=\left\{\begin{array}{ccc}\epsilon_i=\pm1 &\text{for}&1\leq i= j\leq n\\0&\text{for}&i\neq j\end{array}\right.,
&\left.\begin{array}{l}p \text{ of the } \epsilon_i\text{'s are }1 \\q \text{ of the } \epsilon_i\text{'s are }-1\end{array}\right\}.
\end{align} 
To assemble a basis of $\V=\C^{n+4}$ adapted to $\boldsymbol{b}$ and $\boldsymbol{h}=\K^t\boldsymbol{b}$, begin with $\nu,\K\nu\in\V$ spanning a totally null 2-plane, add to these mutually orthogonal unit vectors $\upsilon_i\in\V$ -- i.e., $\boldsymbol{b}(\upsilon_i,\upsilon_j)=\boldsymbol{h}(\upsilon_i,\upsilon_j)=\epsilon_{ij}$ -- and name $\nu',\K\nu'$ the $\boldsymbol{h}$-duals of $\nu,\K\nu$. With respect to the ordered basis $\nu,\K\nu,\upsilon_i,\nu',\K\nu'\in\V$, our bilinear and Hermitian forms are represented
\begin{align}\label{SQbh}
&\boldsymbol{b}=\begin{bmatrix}0&0&\sigma_1\\0&\epsilon_{ij}&0\\\sigma_1&0&0\end{bmatrix},
&\boldsymbol{h}=\begin{bmatrix}0&0&\sigma_0\\0&\epsilon_{ij}&0\\\sigma_0&0&0\end{bmatrix}.
\end{align} 
Let $\B$ consist of all such adapted bases of $\V$. With the standard basis of $V$ serving as the identity element, $\B$ is identified with the Lie group $G(\mathcal{V},b_{\cancel{\HH}})=O(p+2,q+2)$ whose Lie algebra is $\mathfrak{g}=\mathfrak{so}(p+2,q+2)$. In our representation, $\mathfrak{g}$ is $(n+4)\times(n+4)$ matrices of the form
\begin{align}\label{splitLAmat}
&\begin{bmatrix}
\varsigma&\overline{\kappa}&\im\epsilon_{ij}\overline{\zeta}^j&\im\zeta^0&0\\
\kappa&\overline{\varsigma}&-\im\epsilon_{ij}\zeta^j&0&-\im\zeta^0\\
\eta^i&\overline{\eta}^i&\gamma^i_j&\im\zeta^i&-\im\overline{\zeta}^i\\
\im\eta^0&0&-\epsilon_{ij}\overline{\eta}^j&-\overline{\varsigma}&-\overline{\kappa}\\
0&-\im\eta^0&-\epsilon_{ij}\eta^j&-\kappa&-\varsigma
\end{bmatrix},
&\begin{array}{c}
\eta^0,\gamma^i_j,\zeta^0\in\R,\\
\\
\varsigma,\kappa,\eta^i,\zeta^i\in\C,\\
\\
\epsilon_{ij}\gamma^i_k+\epsilon_{ik}\gamma^i_j=0.
\end{array}
\end{align}
Here, we've used summation convention to write $\epsilon_{ij}\eta^j$, which would otherwise be $\epsilon_i\eta^i$ for each fixed $i$. We adhere to the summation convention throughout the paper.  

The symbols $\nu,\K\nu,\upsilon_i,\nu',\K\nu'$ will continue to denote the vectors of a general basis in $\B$, as well as the smooth, $\V$-valued functions on $\B$ which map each basis to the specified vector in it. These functions are differentiated using the $\mathfrak{g}$-valued Maurer-Cartan form on $\B$, represented by \eqref{splitLAmat} whose matrix entries now taken to be 1-forms on $\B$; e.g.,
\begin{align}\label{dnu}
\dd\nu&=\varsigma\otimes\nu+\kappa\otimes\K\nu+\eta^i\otimes\upsilon_i+\eta^0\otimes\im\nu'\in\Omega^1(\B,\V).
\end{align}
The Maurer-Cartan forms themselves are then differentiated according to the Maurer-Cartan equations,
\begin{equation}\label{splitMCeq}
\begin{aligned}
\dd\eta^0&=(\varsigma+\overline{\varsigma})\wedge\eta^0+\im\epsilon_{ij}\eta^i\wedge\overline{\eta}^j,\\
\dd\eta^i&=\zeta^i\wedge\eta^0+\varsigma\wedge\eta^i-\gamma^i_j\wedge\eta^j+\kappa\wedge\overline{\eta}^i,\\
\dd\kappa&=(\varsigma-\overline{\varsigma})\wedge\kappa+\im\epsilon_{ij}\zeta^i\wedge\eta^j,\\
\dd\gamma^i_j&=-\gamma^i_k\wedge\gamma^k_j,\\
\dd\varsigma&=\zeta^0\wedge\eta^0+\im\epsilon_{ij}\eta^i\wedge\overline{\zeta}^j+\kappa\wedge\overline{\kappa},\\
\dd\zeta^i&=\zeta^0\wedge\eta^i-\gamma^i_j\wedge\zeta^j-\overline{\varsigma}\wedge\zeta^i+\kappa\wedge\overline{\zeta}^i,\\
\dd\zeta^0&=-(\varsigma+\overline{\varsigma})\wedge\zeta^0-\im\epsilon_{ij}\zeta^i\wedge\overline{\zeta}^j.
\end{aligned}
\end{equation}

Name the two null cones
\begin{align*}
\N_{\boldsymbol{b}}&=\{v_0\in\V\ |\ \boldsymbol{b}(v_0,v_0)=0\}
&\Rightarrow
&&T_{v_0}\N_{\boldsymbol{b}}=\{v\in\V\ |\ \boldsymbol{b}(v_0,v)=0\},\\
\N_{\boldsymbol{h}}&=\{v_0\in\V\ |\ \boldsymbol{h}(v_0,v_0)=0\}
&\Rightarrow
&&T_{v_0}\N_{\boldsymbol{h}}=\{v\in\V\ |\ \Re\boldsymbol{h}(v_0,v)=0\},
\end{align*}
and let $\N$ be their intersection,
\begin{align*}
&\N=\N_{\boldsymbol{b}}\cap\N_{{\boldsymbol{h}}}
&(\dim_\R\N=2n+5).
\end{align*}
Over each $v_0\in\N$ there are complex subbundles
\begin{align*}
\C\{v_0\}\subset\ker\boldsymbol{h}(v_0,\cdot)\subset T_{v_0}\N,
\end{align*}
the latter having real corank one in $T\N$. Let $\U$ be the image of $\N$ under complex projectivization $\PP:\V\to\C\PP^{n+3}$,
\begin{align*}
&\U=\PP(\N),
&(\dim_\R\U=2n+3),
\end{align*} 
and name its corank-1, complex distribution
\begin{align}\label{splitD}
D_{\PP(v_0)}=\PP_*\ker\boldsymbol{h}(v_0,\cdot)\subset T_{\PP(v_0)}\U.
\end{align}

$\B$ fibers over $\N$ and $\U$ via the projections
\begin{equation}\label{splitframeproj}
\begin{aligned}
\B&\to\N\to\U\\
(\nu,\K\nu,\upsilon_i,\nu',\K\nu')&\mapsto \nu\mapsto\PP\nu,
\end{aligned}
\end{equation}
and we have
\begin{align*}
T_{\nu}\N=\underbrace{\C\{\nu,\K\nu,\upsilon_i\}}_{\ker\boldsymbol{h}(\nu,\cdot)}\oplus\R\{\im\nu'\}
\end{align*}
for each basis in the fiber of \eqref{splitframeproj}. In this sense $\B$ is an \emph{adapted (co)frame bundle} of $\U$: by composing with a local section $s:\U\to\B$, $\nu\in C^\infty(\B,\V)$ is a local section of the tautological line bundle on $\U\subset\PP\V$ while $\K\nu,\upsilon_i\in C^\infty(\B,\V)$ are vector fields whose $\PP_*$-images frame the distribution \eqref{splitD}, and $\im\nu'\in C^\infty(\B,\V)$ completes a local framing of $T\U$. Dually, with a section $s:\U\to\B$ we can pull back the Maurer-Cartan forms on $\B$ to get a local coframing
\begin{align*}
&s^*\eta^0,s^*\eta^i,s^*\kappa\in\Omega^1(\U,\C),
&s^*\eta^0\in\Gamma(D^\bot),
\end{align*}
where the latter means that $s^*\eta^0$ is a section of the annihilator of $D$ in $T^*\U$, as implied by \eqref{dnu}. The first of the Maurer-Cartan equations \eqref{splitMCeq} then shows that $D$ not integrable, as 
\begin{align}\label{homogleviform}
s^*\dd\eta^0\equiv\im\epsilon_{ij}s^*\eta^i\wedge s^*\overline{\eta}^j\mod s^*\eta^0.
\end{align}
Indeed, \eqref{homogleviform} is a local representation of the \emph{Levi form} of $\U$, which measures the failure of the sheaf $\Gamma(D)$ of local sections of $D$ to be closed under the Lie bracket of vector fields. 

In this setting, the Levi form coincides with the Hermitian form $\boldsymbol{h}$ acting on the vector fields $\K\nu,\upsilon_i$ which locally frame $D$ (henceforth, we elide $s$ in discussing local (co)framings of $\U$). In particular, Lie brackets of the $\upsilon_i$ are no longer sections of $D$ as $\boldsymbol{h}(\upsilon_i,\upsilon_j)=\epsilon_{ij}\neq0$. By contrast, the real and imaginary parts of $\K\nu$ span a rank-2 subbundle of $D$, the \emph{Levi kernel}, which is integrable by virtue of $\boldsymbol{h}(\K\nu,\K\nu)=0$ and the Newlander-Nirenberg Theorem. As such, the Levi kernel is tangent to the leaves of a foliation of $\U$ by complex curves, the \emph{Levi foliation}. The foregoing argument is the vector-field equivalent of the observation that 
\begin{align}\label{splitCPline}
&\text{for }(\nu,\K\nu,\upsilon_i,\nu',\K\nu')\in\B,
&\{\PP(\nu+c\K\nu)\ |\ c\in\C\}\subset\U
&&\text{is a complex-projective line.}
\end{align}
There is a natural equivalence relation on $\U$:
\begin{align*}
&\widetilde{u}_1\sim \widetilde{u}_2
&\Longleftrightarrow
&&\widetilde{u}_1,\widetilde{u}_2\text{ lie in the same line }\eqref{splitCPline};
&&(\widetilde{u}_1,\widetilde{u}_2\in\U).
\end{align*}
The \emph{leaf space} of the Levi foliation of $\U$ is the image of the canonical quotient projection 
\begin{align*}
&Q:\U\to\U/\sim,
&U=Q(\U)
&&(\dim_\R U=2n+1).
\end{align*}
As we saw in \eqref{Aspanv}, \eqref{AHH}, the projective line \eqref{splitCPline} is the split-quaternionic line $\nu\cancel{\HH}\subset\U$, hence $U=\mathsf{Gr}_2^0\V$, the Grassmannian of totally $\boldsymbol{b},\boldsymbol{h}$-null complex 2-planes in $\V$. Complex-scalar multiples of $\K\nu$ are the fibers of $Q:\U\to U$, but local trivializations $\U\approx U\times\C$ of this fibration necessarily fail to preserve the geometry of $\U$, which we discern from the Maurer-Cartan equations \eqref{splitMCeq} as follows.  

It's clear from \eqref{dnu} and \eqref{splitCPline} that the 1-form $\kappa\in\Omega^1(\U,\C)$ measures the infinitesimal variation of $\nu\in\U$ in the direction of $\K\nu$, whose triviality under the Levi form is evidenced by the absence of $\kappa$ (or $\overline{\kappa}$) in \eqref{homogleviform}. Subsequent Maurer-Cartan equations read
\begin{align}\label{splithomog2non}
\dd\eta^i\equiv \kappa\wedge\overline{\eta}^i\mod\{\eta^0,\eta^j\},
\end{align}
which shows that $\U$ is \emph{2-nondegenerate}; i.e., there is no local diffeomorphism $\U\to U\times\C$ whose pushforward preserves the complex structure on $D\subset T\U$. In other words, Lie derivatives of (sections of) $D$ with respect to (sections of) the Levi kernel vary the fibers of $D$ along those of $Q:\U\to U$ in such a way that $Q_*D\subset TU$ has no canonical complex structure.  

We conclude this section by comparing $\U=\PP(\N)$ and $U=\mathsf{Gr}_2^0\V$ as homogeneous quotients of $\B\cong O(p+2,q+2)$. For $(\nu,\K\nu,\upsilon_i,\nu',\K\nu')\in\B$ we name the stabilizer subgroups of $O(p+2,q+2)$:
\begin{equation}
\begin{aligned}
&\mathcal{R}_1\subset O(p+2,q+2)
&\text{stabilizing the complex line}
&&\C\{\nu\}\in\PP\V,\\
&\mathcal{R}_2\subset O(p+2,q+2)
&\text{stabilizing the complex 2-plane}
&&\C\{\nu,\K\nu\}\in\mathsf{Gr}_2^0\V,
\end{aligned}
\end{equation}
so that $\U=O(p+2,q+2)/\mathcal{R}_1$ and $U=O(p+2,q+2)/\mathcal{R}_2$. Roughly illustrated,
\begin{align}\label{stabsubgps}
&\mathcal{R}_1=\begin{bmatrix}*&*&*\cdots*\\0&*&*\cdots*\\0&*&*\cdots*\\\vdots&\vdots&*\cdots*\\0&*&*\cdots*\end{bmatrix},
&\mathcal{R}_2=\begin{bmatrix}*&*&*\cdots*\\ * &*&*\cdots*\\0&0&*\cdots*\\\vdots&\vdots&*\cdots*\\0&0&*\cdots*\end{bmatrix}
&&\subset O(p+2,q+2),
\end{align}
whence \eqref{splitLAmat} shows that the Levi kernel 1-form $\kappa$ is semibasic for $\B\to\U$ and vertical for $\B\to U$.

\vspace{\baselineskip}

\subsection{Quaternionic Lines as Leaves of a 2-Nondegenerate Levi Foliation}\label{Qlinessec}

The discussion in \S\ref{SQlinessec} carries over to the quaternionic case \emph{mutatis mutandis}, so we record what mutates. For $p\geq 1$, let $n=2p$ and set $V=\R^{n+4}$ as before, with bilinear form and quaternionic structure
\begin{align*}
&b=\begin{bmatrix}0&0&0&\sigma_1\\0&0&\mathbbm{1}_p&0\\0&\mathbbm{1}_p&0&0\\\sigma_1&0&0&0\end{bmatrix},
&\K=\begin{bmatrix}-\im\sigma_2&0&0&0\\0&0&-\mathbbm{1}_p&0\\0&\mathbbm{1}_p&0&0\\0&0&0&-\im\sigma_2\end{bmatrix},
\end{align*}
so that the Lie algebra $\mathfrak{g}=\mathfrak{so}^*(2p+4)$ is represented
\begin{align}\label{QLAmat}
&\begin{bmatrix}
\varsigma&-\overline{\kappa}&\im\overline{\zeta}^i&\im\overline{\zeta}^{p+i}&\im\zeta^0&0\\
\kappa&\overline{\varsigma}&\im\zeta^{p+i}&-\im\zeta^{i}&0&-\im\zeta^0\\
\eta^i&-\overline{\eta}^{p+i}&\xi^i_j &\xi^i_{p+j} &\im\zeta^i&\im\overline{\zeta}^{p+i}\\
\eta^{p+i}&\overline{\eta}^i&\xi^{p+i}_j &\xi^{p+i}_{p+j} &\im\zeta^{p+i}&-\im\overline{\zeta}^{i}\\
\im\eta^0&0&-\overline{\eta}^i&\overline{\eta}^{p+i}&-\overline{\varsigma}&\overline{\kappa}\\
0&-\im\eta^0&-\eta^{p+i}&-\eta^{i}&-\kappa&-\varsigma
\end{bmatrix}
&\begin{array}{cc}\xi^i_j,\xi^{p+i}_j,\xi^i_{p+j},\xi^{p+i}_{p+j}\in\C,&1\leq i\leq p,\\ \\
\xi^i_j+\overline{\xi}^j_i=0, &\xi^{p+i}_{p+j}=\overline{\xi}^i_j,\\ \\
\xi^{p+i}_j+\xi^{p+j}_i=0, &\xi^i_{p+j}=-\overline{\xi}^{p+i}_j.
\end{array}
\end{align}
We highlight a few of the Maurer-Cartan equations,
\begin{equation}\label{quatMCeq}
\begin{aligned}
\dd\eta^0&=(\varsigma+\overline{\varsigma})\wedge\eta^0+\im\delta_{ij}\eta^i\wedge\overline{\eta}^j-\im\delta_{ij}\eta^{p+i}\wedge\overline{\eta}^{p+j},\\
\dd\eta^i&=\zeta^i\wedge\eta^0+\varsigma\wedge\eta^i-\xi^i_j\wedge\eta^j-\xi^i_{p+j}\wedge\eta^{p+j}-\kappa\wedge\overline{\eta}^{p+i},\\
\dd\eta^{p+i}&=\zeta^{p+i}\wedge\eta^0+\varsigma\wedge\eta^{p+i}-\xi^{p+i}_j\wedge\eta^j-\xi^{p+i}_{p+j}\wedge\eta^{p+j}+\kappa\wedge\overline{\eta}^i.
\end{aligned}
\end{equation}
In particular, the Levi form \eqref{homogleviform} of $\U$ matches \eqref{epsilonmetric} for $q=p$ and $\epsilon_i=1=-\epsilon_{p+i}$ (with all due apology, we let $i\leq \tfrac{1}{2}n$ in the quaternionic setting). Moreover, the 2-nondegeneracy of $\U$ which was evinced by \eqref{splithomog2non} for split-quaternionic lines is revealed here by
\begin{align}\label{quathomog2non}
\dd\begin{bmatrix}\eta^i\\\eta^{p+i}\end{bmatrix}\equiv
\begin{bmatrix}0&-\kappa\\\kappa&0\end{bmatrix}\wedge\begin{bmatrix}\overline{\eta}^i\\\overline{\eta}^{p+i}\end{bmatrix}
\mod\{\eta^0,\eta^j,\eta^{p+j}\}.
\end{align}
In the felicitous language of the next section, the discrepancy between \eqref{splithomog2non} and \eqref{quathomog2non} is explained by the fact that the leaf space $U$ of $\U$ carries an almost-split-quaternionic structure in \S\ref{SQlinessec} and an almost-quaternionic structure in \S\ref{Qlinessec}.

\vspace{\baselineskip}


\section{L-contact Manifolds}\label{Lconsec}


\subsection{Almost (Hyper)CR Structures}\label{AHCsec}

Let $\U$ be a smooth manifold. For any fiber bundle $E\to \U$, $E_o$ denotes the fiber of $E$ over $o\in\U$, and $\Gamma(E)$ is the sheaf of smooth local sections. If $E$ is a vector bundle, $E^*$ is its dual and $\C E$ its complexification, with fibers $\C E_o=E_o\otimes_\R\C=E_o\oplus\im E_o$ ($\im=\sqrt{-1}$), and $\mathbbm{1}$ denotes the identity endomorphism field on $E$ or $\C E$.

An \emph{almost-complex structure} on an even-rank distribution $D\subset T\U$ is an endomorphism field 
\begin{align}\label{ACdef}
&J:D\to D
&\text{satisfying}
&&J^2=-\mathbbm{1}.
\end{align}
An almost-complex structure determines a splitting
\begin{align}\label{ACLambdasplit}
\C D=\Lambda\oplus\overline{\Lambda}
\end{align}
into $\pm\im$-eigenspaces of $J$,
\begin{align}\label{Lambdadef}
&\Lambda=\{X-\im JX\ |\ X\in D\},
&\overline{\Lambda}=\{X+\im JX\ |\ X\in D\}
&&\subset\C D.
\end{align}
For $\varepsilon=\pm1$, an \emph{almost-$\varepsilon$-hypercomplex structure} is an almost-complex structure as in \eqref{ACdef} along with a second endomorphism field 
\begin{align}\label{AeQdef}
&K:D\to D
&\text{satisfying}
&&J\circ K=-K\circ J,
&&K^2=\varepsilon\mathbbm{1}.
\end{align}
In view of \eqref{Lambdadef}, it is apparent that the $\C$-linear extension $K:\C D\to\C D$ echoes the \emph{conjugate}-linear $\varepsilon$-quaternionic map $\K:\V\to\V$ of \S\ref{splitquatsec} in that 
\begin{align*}
&K:\overline{\Lambda}\to\Lambda,
&K:\Lambda\to\overline{\Lambda},
&&K(\overline{Y})=\overline{K(Y)}\quad \forall Y\in\Lambda.
\end{align*}
Thus, almost-$(-1)$-hypercomplex structures require that $\mathrm{rank}_\C\Lambda=\tfrac{1}{2}\mathrm{rank}D$ is even. 

\begin{rem}\label{HCtermrem}
Typically, ``almost-(hyper)complex structures" refer to those defined on $D=T\U$ for $\dim\U$ even, and $\U$ is a (hyper)complex manifold if the endomorphisms $J,K$ satisfy some additional integrability conditions. For proper subbundles $D\subsetneq T\U$, the term ``CR" takes the place of ``complex," subject to some smoothness and integrability considerations for $D$. Regarding definition \ref{AeQdef}, we should also mention that the standard parlance treats (almost) hypercomplex structures as special cases of (almost) quaternionic structures (see \cite[\S4.1.8]{CapSlovak}, which only pertains to $\varepsilon=-1$), the latter characterized by a 3-dimensional subbundle of $D^*\otimes D$ which admits \emph{local} framing by some $\{J,K,JK\}$, whereas ``hypercomplex" is reserved for distinguished, global $J,K$. 
\end{rem}

Let $\U$ have odd dimension $\geq 3$ and a smooth, corank-1 distribution $D\subset T\U$. $D$ is integrable in the Frobenius sense if its sections are closed under the Lie bracket of vector fields, $[\Gamma(D),\Gamma(D)]\subset\Gamma(D)$. The \emph{complexified Levi bracket} (cf. \cite[Def. 3.1.7]{CapSlovak}) 
\begin{align}\label{Levidom}
\mathcal{L}:\C D\times\C D\to\C(T\U/D)
\end{align}
measures the failure of $D$ to be integrable,
\begin{equation}\label{Levidef}
\mathcal{L}(y_1,y_2)=\im[Y_1,Y_2](o)\mod\C D_o\left\{\begin{array}{c} y_1,y_2\in\C D_o,\\Y_1,Y_2\in\Gamma(\C D),\\ Y_1(o)=y_1,\ Y_2(o)=y_2.\end{array}\right.
\end{equation}
Because it takes values in the quotient $\C(T\U/D)$, $\mathcal{L}$ is tensorial and, up to a choice of local trivialization $\C(T\U/D)\approx\C$, may be considered a skew-symmetric, $\C$-bilinear form on $\C D$.

An \emph{almost-CR structure} $J:D\to D$ is an almost-complex structure \eqref{ACdef} which satisfies the \emph{partial-integrability} condition,
\begin{align}\label{partialintLJ}
&\mathcal{L}(JX,JY)=\mathcal{L}(X,Y)
&\Longleftrightarrow 
&&\mathcal{L}(JX,Y)+\mathcal{L}(X,JY)=0
&&\forall X,Y\in\Gamma(D). 
\end{align}
In terms of \eqref{Lambdadef}, \eqref{partialintLJ} can be rephrased
\begin{align}\label{partialint}
&[\Gamma(\Lambda),\Gamma(\Lambda)]\subset\Gamma(\C D)
&\Longleftrightarrow 
&&\mathcal{L}|_{\Lambda\times\Lambda}=0=\mathcal{L}|_{\overline{\Lambda}\times\overline{\Lambda}},
\end{align} 
so that \eqref{ACLambdasplit} is an $\mathcal{L}$-null splitting.

\begin{definition}\label{CRdef} A \emph{CR structure} on an odd-dimensional manifold $\U$ with a corank-1 distribution $D\subset T\U$ is a splitting $\C D=H\oplus\overline{H}$ satisfying the \emph{CR integrability} condition:
\begin{equation}\label{CRint}
[\Gamma(H),\Gamma(H)]\subset\Gamma(H).
\end{equation}
\end{definition}

\begin{rem}\label{Nijenhuisrem}
A CR structure determines an almost-CR structure as follows: for $X\in D$, $X=\Re Z$ for some $Z\in H$, and we can define $JX=-\Im Z$ (where $\Re Z=\tfrac{1}{2}(Z+\overline{Z})$, $\Im Z=-\tfrac{\im}{2}(Z-\overline{Z})$). Thus, the distinction between an almost-CR structure and a CR structure as we've defined them  -- aside from the choice of labels $\Lambda$ and $H$ for the $\im$-eigenspaces -- is exactly the distinction between partial integrability \eqref{partialint} and CR integrability \eqref{CRint}, stated in their complex-conjugate forms as 
\begin{align*}
&[\Gamma(\overline{\Lambda}),\Gamma(\overline{\Lambda})]\subset\Gamma(\Lambda\oplus\overline{\Lambda})
&\text{vs}
&&[\Gamma(\overline{H}),\Gamma(\overline{H})]\subset\Gamma(\overline{H}).
\end{align*}
This difference is quantified by the \emph{Nijenhuis tensor} $N_\Lambda:\Lambda\wedge\Lambda\to\Lambda$ of $\Lambda$:
\begin{align*}
&N_\Lambda(Y_1,Y_2)=[\overline{Y}_1,\overline{Y}_2]\mod\overline{\Lambda},
&Y_1,Y_2\in\Gamma(\Lambda),
\end{align*}
which obviously vanishes for $H$. As discussed in \cite[\S4.2.4]{CapSlovak} (in the Levi-nondegenerate case), there is no need to exclude partially-integrable structures from the CR category, as they can simply be considered ``CR structures with torsion" -- the harmonic component of which is the Nijenhuis tensor -- and the usual classification scheme of Cartan's method of equivalence still applies. We choose to name and label them differently because in \S\ref{2nonsec} we consider Levi-degenerate (CR integrable) structures on $\U$ that give rise to a complex-parameter-family of (partially integrable) almost-CR structures on the leaf space of $\U$, so we hope that investing in the distinction now pays off in conceptual clarity later.   
\end{rem}

An \emph{almost-$\varepsilon$-hyper-CR structure}\footnote{Hyper-CR structures are defined differently in \cite{hyperCR}, extending the terminology of \cite{hyperherm}, where ``hyper-Hermitian" refers to hypercomplex structures (as in Remark \ref{HCtermrem}) that are orthogonal for some metric. Our usage does not contradict that of \cite{hyperCR}; in fact, the hyper-CR structure of a unit tangent bundle (\S\ref{UTBsec}) is orthogonal for the Sasaki metric.} on a manifold equipped with an almost-CR structure \eqref{partialintLJ} is another endomorphism field $K:D\to D$ satisfying
\begin{align}\label{ASQdef}
&K^2=\varepsilon\mathbbm{1},
&KJ=-JK,
&&\mathcal{L}(KX,Y)+\mathcal{L}(X,KY)=0\quad\forall X,Y\in\Gamma(D).
\end{align} 
As in \S\ref{splitquatsec}, $K$ will be called \emph{split-quaternionic} when $\varepsilon=1$ and \emph{quaternionic} when $\varepsilon=-1$. To register a final analogy between $\K$ and $K$, note that \eqref{CbK} implies the Hermitian identity
\begin{align*}
&\boldsymbol{h}(\K v_1,\K v_2)
=\varepsilon\overline{\boldsymbol{h}(v_1,v_2)}
&v_1,v_2\in\V,
\end{align*}
whereas \eqref{Levidef} and \eqref{ASQdef} show
\begin{align}\label{ASQherm}
&\mathcal{L}(K\overline{Y}_1,K\overline{Y}_2)
=\varepsilon\overline{\mathcal{L}(Y_1,Y_2)},
&Y_1,Y_2\in\C D.
\end{align}

\vspace{\baselineskip}

\subsection{2-Nondegenerate CR Manifolds and Their Leaf Spaces}\label{2nonsec}

Let $\U$ be a smooth manifold of dimension $2m+1$ equipped with a CR structure as in Definition \ref{CRdef}; i.e., a corank-1 subbundle $D^{2m}\subset T\U$ whose complexification splits into the CR subbundle and its complex conjugate (anti-CR) bundle,
\begin{align*}
&\C D=H\oplus\overline{H}
&(\mathrm{rank}_\C H=\mathrm{rank}_\C\overline{H}=m),
\end{align*}
satisfying the CR integrability condition \eqref{CRint}. In particular, $H$ and $\overline{H}$ Levi-null, 
\begin{equation*}
\mathcal{L}|_{H\times H}=0=\mathcal{L}|_{\overline{H}\times\overline{H}},
\end{equation*}
where $\mathcal{L}$ is the Levi bracket \eqref{Levidom}, \eqref{Levidef}. The \emph{Levi kernel} $\LK\subset H$ is
\begin{align*}
\LK=\{X\in H\ |\ \mathcal{L}(X,Y)=0\ \forall Y\in\C D\},
\end{align*}
and we assume $\mathrm{rank}_\C\LK=k>0$ is constant. By the Newlander-Nirenberg Theorem, $\U$ is foliated by complex submanifolds of complex dimension $k$ which are tangent to $\LK\oplus\overline{\LK}$, the \emph{Levi-foliation}. The Levi-foliation introduces an equivalence relation on $\U$,
\begin{align*}
&\widetilde{u}_1\sim \widetilde{u}_2
&\Longleftrightarrow
&&\widetilde{u}_1,\widetilde{u}_2\text{ lie in the same leaf of the Levi-foliation};
&&(\widetilde{u}_1,\widetilde{u}_2\in\U).
\end{align*}
The \emph{leaf space} of $\U$ is the image of the canonical quotient map $Q:\U\to U=\U/\sim$. $U$ is a smooth manifold of dimension $2n+1$ where $n=m-k$. Whether a bundle or tensor on $\U$ descends along $Q$ to be well-defined on $U$ depends on its behavior subject to the Lie derivative with respect to the Levi kernel; e.g., 
\begin{align*}
[\Gamma(\LK),\Gamma(\C D)]\subset\Gamma(\C D),
\end{align*}
hence $\Delta=Q_*D\subset TU$ is a well-defined contact distribution on the leaf space. Similarly, the Levi form $\mathcal{L}$ of $\U$ descends to a symplectic form $\mathfrak{L}$ on $\C\Delta\cong\C D/(\LK\oplus\overline{\LK})$. 

However, the leaf space does not necessarily inherit a CR structure from $\U$, as $[\Gamma(\LK),\Gamma(\overline{H})]\not\subset \Gamma(\LK\oplus\overline{H})$ in general. More precisely, for each $\widetilde{u}\in\U$ in a given leaf $Q(\widetilde{u})=u\in U$, we have complementary $\mathfrak{L}$-Lagrangian subspaces
\begin{align}\label{leaflag}
&\Lambda_{\widetilde{u}}=Q_*(H_{\widetilde{u}})\cong H_{\widetilde{u}}/\LK_{\widetilde{u}},
&\overline{\Lambda}_{\widetilde{u}}=Q_*(\overline{H}_{\widetilde{u}})\cong\overline{H}_{\widetilde{u}}/\overline{\LK}_{\widetilde{u}}
&&\subset\C\Delta_u,
\end{align}
and integrability of the Levi kernel ensures the Lie derivatives $[X,\overline{Y}]$, $[\overline{X},Y]$ ($X\in\Gamma(\LK), Y\in\Gamma(H)$) descend to well-defined, tensorial operators
\begin{align*}
&\mathrm{ad}_{X_{\widetilde{u}}}:\left\{\begin{array}{l}\overline{\Lambda}_{\widetilde{u}}\longrightarrow\Lambda_{\widetilde{u}}\\\Lambda_{\widetilde{u}}\stackrel{\text{def}}{\longrightarrow}0\end{array}\right.,
&\mathrm{ad}_{\overline{X}_{\widetilde{u}}}:\left\{\begin{array}{l}\Lambda_{\widetilde{u}}\longrightarrow\overline{\Lambda}_{\widetilde{u}}\\\overline{\Lambda}_{\widetilde{u}}\stackrel{\text{def}}{\longrightarrow}0\end{array}\right.,
&&X_{\widetilde{u}}\in \LK_{\widetilde{u}},
\end{align*}
which we have extended trivially to the rest of $\C\Delta$ as \eqref{CRint} obviates the need to consider $[\LK,H]$ or $[\overline{\LK},\overline{H}]$. $\U$ is called \emph{straightenable} if $\mathrm{ad}_X=0$ $\forall X\in \LK$, meaning the subspaces $\{\Lambda_{\widetilde{u}}\subset\C\Delta_{u}\ |\ \widetilde{u}\in u\}$ are all canonically identified and thus determine a well-defined CR structure on the leaf space $U$. The opposite extreme is $\mathrm{ad}_X\neq 0$ for every nonzero $X\in \LK$, in which case $U$ has no canonical CR structure and $\U$ is called \emph{2-nondegenerate}.

In any case, the maps $\mathrm{ad}_\LK=\{\mathrm{ad}_X\ |\ X\in \LK\}$ and $\mathrm{ad}_{\overline{\LK}}$ serve to quantify the obstruction to straightenability of $\U$ in a manner that we will exploit geometrically. Varying $\widetilde{u}$ smoothly within a fixed leaf $Q(\widetilde{u})=u\in U$ will vary the Lagrangian subspaces \eqref{leaflag} of $\C\Delta_{u}$, sweeping out two submanifolds of the Lagrangian-Grassmannian,
\begin{align}\label{LconjL}
&L_u=\{\Lambda_{\widetilde{u}}\subset\C\Delta_{u}\ |\ \widetilde{u}\in u\},
&\overline{L}_u=\{\overline{\Lambda}_{\widetilde{u}}\subset\C\Delta_{u}\ |\ \widetilde{u}\in u\}
&&\subset\mathsf{LaGr}(\C\Delta_{u}).
\end{align}
If $\U$ is straightenable then $L_u$ and $\overline{L}_u$ are singletons, but in the 2-nondegenerate setting, they encode the Levi foliation of $\U$ within the bundle $\mathsf{LaGr}(\C\Delta)\to U$ over the leaf space, as we now demonstrate. 

The Jacobi identity for the Lie bracket of vector fields implies 
\begin{align}\label{adKder}
&\mathfrak{L}(\mathrm{ad}_{\overline{X}}(Y_1),Y_2)+\mathfrak{L}(Y_1,\mathrm{ad}_{\overline{X}}(Y_2))=0,
&\forall X\in \LK_{\widetilde{u}},\ Y_1,Y_2\in\Lambda_{\widetilde{u}}\oplus\overline{\Lambda}_{\widetilde{u}}.
\end{align} 
As $\mathfrak{L}$-Lagrangian complements, $\Lambda_{\widetilde{u}}$ and $\overline{\Lambda}_{\widetilde{u}}$ and may be regarded as each other's dual spaces, whence 
\begin{align*}
&\mathrm{ad}_{\overline{\LK}_{\widetilde{u}}}\subset S^2(\Lambda_{\widetilde{u}})^*\cong T_{\Lambda_{\widetilde{u}}}\mathsf{LaGr}(\C\Delta_u)\subset\text{\sc{Hom}}(\Lambda_{\widetilde{u}},\overline{\Lambda}_{\widetilde{u}}),\\
&\mathrm{ad}_{\LK_{\widetilde{u}}}\subset S^2(\overline{\Lambda}_{\widetilde{u}})^*\cong T_{\overline{\Lambda}_{\widetilde{u}}}\mathsf{LaGr}(\C\Delta_u)\subset\text{\sc{Hom}}(\overline{\Lambda}_{\widetilde{u}},\Lambda_{\widetilde{u}}).
\end{align*}
When $\U$ is 2-nondegenerate, $\mathrm{ad}:\LK\to\mathrm{ad}_\LK$ is injective and serves to identify $\LK$ and $\overline{\LK}$ with the tangent bundles of their respective factors in the product space
\begin{align*}
\overline{L}_u\times L_u=\{(\overline{\Lambda}_{\widetilde{u}_1},\Lambda_{\widetilde{u}_2})\ |\ \widetilde{u}_1,\widetilde{u}_2\in u\}\subset\mathsf{LaGr}(\C\Delta_u)\times\mathsf{LaGr}(\C\Delta_u),
\end{align*}
which features an involution combining transposition and conjugation,
\begin{align*}
\overline{\tau}:\overline{L}_u\times L_u&\to \overline{L}_u\times L_u\\
(\overline{\Lambda}_{\widetilde{u}_1},\Lambda_{\widetilde{u}_2})&\mapsto (\overline{\Lambda}_{\widetilde{u}_2},\Lambda_{\widetilde{u}_1})
\end{align*}
whose fixed-point set $\mathbb{L}_u=\{(\overline{\Lambda}_{\widetilde{u}},\Lambda_{\widetilde{u}})\in \overline{L}_u\times L_u\ |\ \widetilde{u}\in u\}$ is a submanifold of $\mathsf{LaGr}(\C\Delta_u)$ sitting diagonally in $\mathsf{LaGr}(\C\Delta_u)\times\mathsf{LaGr}(\C\Delta_u)$. Now $\overline{\tau}_*:\LK_{\widetilde{u}_1}\oplus \overline{\LK}_{\widetilde{u}_2}\to \LK_{\widetilde{u}_2}\oplus \overline{\LK}_{\widetilde{u}_1}$ is the identity map on $T\mathbb{L}_u$, which is therefore $\Re(\LK\oplus\overline{\LK})$.

The total space $\mathbb{L}=\{\mathbb{L}_u\ |\ u\in U\}$ defines a bundle over the leaf space from which $\U$ is \emph{recoverable}, in the language of \cite{SykesZelenko} (at least under some additional hypotheses on $\mathrm{ad}_\LK$ which will always be satisfied for our purposes). Generalizing this construction enables us to circumvent $\U$ while we investigate the geometry of $U$.  

\begin{definition}\label{DLCdef}(cf. \cite[Def. 2.3]{SykesZelenko})
Let $U$ be a smooth manifold of dimension $2n+1$. A \emph{dynamical Legendrian-contact structure} on $U$ consists of: 
\begin{enumerate}[label=\textbf{\arabic*.}]

\item\label{DLCDeltaLevi} a contact distribution $\Delta\subset TU$ with ``Levi form" $\mathfrak{L}:\C\Delta\times\C\Delta\to\C(TU/\Delta)$ given by 
\begin{align*}
&\mathfrak{L}(X,Y)=\im[X,Y]\mod\C\Delta,
&X,Y\in\Gamma(\C\Delta);
\end{align*}

\item\label{DLCLG} $\mathsf{LaGr}(\C\Delta)\to U$ whose fiber over $u\in U$ is 
\begin{align*}
\mathsf{LaGr}(\C\Delta_u)=\{\Lambda\subset\C\Delta_u\ |\ \dim_\C\Lambda=n,\ \mathfrak{L}(v_1,v_2)=0\ \forall v_1,v_2\in\Lambda\};
\end{align*}

\item the product bundle $\mathsf{LaGr}(\C\Delta)\times \mathsf{LaGr}(\C\Delta)\to U$ with two involutions
\begin{align*}
\overline{\tau},\tau:\mathsf{LaGr}(\C\Delta_u)\times \mathsf{LaGr}(\C\Delta_u)&\to \mathsf{LaGr}(\C\Delta_u)\times \mathsf{LaGr}(\C\Delta_u)
\end{align*}
given by transposition $\tau(\Lambda_1,\Lambda_2)=(\Lambda_2,\Lambda_1)$ and its composition with complex conjugation with respect to the real subspace $\Delta_u\subset\C\Delta_u$,  $\overline{\tau}(\Lambda_1,\Lambda_2)=(\overline{\Lambda}_2,\overline{\Lambda}_1)$ for $\Lambda_1,\Lambda_2\in\mathsf{LaGr}(\C\Delta_u)$;

\item\label{DLCLmfd} a smooth submanifold $\mathbb{L}\subset\mathsf{LaGr}(\C\Delta)\times \mathsf{LaGr}(\C\Delta)$ such that, $\forall u\in U$,

\begin{enumerate}[label=\textbf{\alph*.}]

\item\label{DLCLC} $\mathbb{L}_u=\mathbb{L}\cap\mathsf{LaGr}(\C\Delta_u)\times\mathsf{LaGr}(\C\Delta_u)$ is a complex manifold of complex dimension $k$;

\item\label{DLCCR} $\tau(\mathbb{L}_u)\cap\mathbb{L}_u=\varnothing$;

\item\label{DLCconj} $\overline{\tau}|_{\mathbb{L}_u}$ is the identity map on $\mathbb{L}_u$.
\end{enumerate}
\end{enumerate}
\end{definition}

\begin{rem}\label{LconK}
It is implicit in the definition that we identify $\C T_{(\overline{\Lambda},\Lambda)}\mathbb{L}_u$ with the subspace of 
\begin{align*}
T_{(\overline{\Lambda},\Lambda)}(\mathsf{LaGr}(\C\Delta_u)\times\mathsf{LaGr}(\C\Delta_u))\subset\text{\sc{Hom}}(\overline{\Lambda},\Lambda)\oplus\text{\sc{Hom}}(\Lambda,\overline{\Lambda})
\end{align*}
that contains $T_{(\overline{\Lambda},\Lambda)}\mathbb{L}$ as its real subspace with respect to the conjugation operator $\overline{\tau}_*$. Thus, we extend the rationale for leaf spaces above to obtain a CR splitting of $\C T_{(\overline{\Lambda},\Lambda)}\mathbb{L}_u$ by taking $\LK$ and $\overline{\LK}$ to be (the complex spans of) the images of the projections $T_{(\overline{\Lambda},\Lambda)}\mathbb{L}\to \text{\sc{Hom}}(\overline{\Lambda},\Lambda)$ and $T_{(\overline{\Lambda},\Lambda)}\mathbb{L}\to \text{\sc{Hom}}(\Lambda,\overline{\Lambda})$, respectively. 
\end{rem}

\begin{rem}\label{homogDLCrem} The fibers $\mathsf{LaGr}(\C\Delta_u)$ of the Lagrangian-Grassmannian bundle are homogeneous for the action of the symplectic group $Sp(\C\Delta_u,\mathfrak{L}_u)\cong Sp(2n)$ (\cite{LGgeom}), and dynamical Legendrian contact structures are especially symmetric when the fibers $\mathbb{L}_u$ are foliated by distinguished curves generated by the action of $Sp(2n)$ (\cite[\S1.4.11]{CapSlovak}). When such $\mathbb{L}_u$ are complex curves ($k=1$ in Def.\ref{DLCdef} \ref{DLCLmfd}\ref{DLCLC}), $\mathbb{L}$ is completely determined by a local section $U\to\mathbb{L}$ (i.e., a choice of splitting $\C\Delta=\Lambda\oplus\overline{\Lambda}$ for $(\overline{\Lambda},\Lambda)\in\mathbb{L})$ along with nonvanishing endomorphism fields $\mathfrak{a}:\overline{\Lambda}\to\Lambda$, $\overline{\mathfrak{a}}:\Lambda\to\overline{\Lambda}$ that span $\LK$ and $\overline{\LK}$ as in Remark \ref{LconK}. In particular, there is a dynamical Legendrian contact structure naturally associated to any contact manifold $(U,\Delta)$ carrying an almost-$\varepsilon$-quaternionic structure $K$ as in \S\ref{AHCsec}, whose bundles $\LK,\overline{\LK}$ share a common fiber coordinate $a\in C^\infty(U,\C)$ with respect to which
\begin{align*}
&\LK=\{aK:\overline{\Lambda}\to\Lambda\},
&\overline{\LK}=\{\overline{a}K:\Lambda\to\overline{\Lambda}\}.
\end{align*}
\end{rem}

To compare to \cite{SykesZelenko}, note that formulating \cite[Def. 2.3]{SykesZelenko} for complex $\C U$ (the analytic continuation of real-analytic $U$) facilitates the ``recovery" process (\cite[Rem. 2.5, Prop. 2.6]{SykesZelenko}) of reconstructing 2-nondegenerate $\U$ from the dynamical Legendrian contact structure of its leaf space. Since we are not concerned with recovering 2-nondegenerate CR manifolds, we formulate Definition \ref{DLCdef} in the smooth (real) category. Still, it is instructive to analyze the definition in light of 2-nondegenerate CR geometry. Condition \ref{DLCLmfd} is automatic if $U$ is the leaf space of a 2-nondegenerate CR manifold; \ref{DLCCR} and \ref{DLCconj} ensure that the two complex submanifolds \eqref{LconjL} are appropriately non-intersecting and conjugate to one another. Item \ref{DLCLC} applies to leaf spaces of 2-nondegenerate CR manifolds whose Levi kernel has complex rank $k$ so that the Levi foliation consists of complex $k$-submanifolds. When $k=1$ so that $\U$ is foliated by complex curves, 2-nondegeneracy is equivalent to non-straightenability (higher nondegeneracy conditions come into play if $k\geq 2$; see \cite[Ch. 11]{Bao}), and this case will be our focus. 

\begin{definition}\label{Lcondef} A dynamical Legendrian contact manifold $U$ as in Definition \ref{DLCdef} is called \emph{L-contact} when 
\begin{enumerate}[label=\textbf{\arabic*.}]
\item\label{Lconcurve} $\mathbb{L}_u$ as in Def.\ref{DLCdef} \ref{DLCLmfd} has complex dimension $k=1$;

\item\label{LconSR} for each $(\overline{\Lambda},\Lambda)\in\mathbb{L}_u$ and $(\mathfrak{a},\overline{\mathfrak{a}})\in T_{(\overline{\Lambda},\Lambda)}\mathbb{L}_u\subset\Re(\text{\sc{Hom}}(\overline{\Lambda},\Lambda)\oplus\text{\sc{Hom}}(\Lambda,\overline{\Lambda}))$, 
\begin{align*}
&\mathfrak{L}(\mathfrak{a}(\overline{v}_1),\overline{\mathfrak{a}}(v_2))=\varepsilon r^2\overline{\mathfrak{L}(v_1,\overline{v}_2)}
&\forall v_1,v_2\in\Lambda,
\end{align*}
for $\varepsilon=\pm1$ and some $r\in\R$ which is zero if and only if $\mathfrak{a}=0$. 
\end{enumerate}
\end{definition}

Item \ref{LconSR} of Definition \ref{Lcondef}, called the ``conformal unitary" condition in \cite{CPCAG} or ``strongly regular" in \cite{PZ}, together with \ref{Lconcurve} situates L-contact structures within the category modeled on the homogeneous spaces presented in \S\ref{SQlinessec} and \S\ref{Qlinessec}. There, the L-contact structure arose from an $\varepsilon$-quaternionic structure; as Def. \ref{Lcondef} \ref{LconSR} and \eqref{ASQherm} suggest, an \emph{almost} $\varepsilon$-quaternionic structure can be similarly instrumental. Indeed, \cite[Cor. 4.6]{PZ} assures us that L-contact manifolds are among the favorable dynamical Legendrian contact structures described in Remark \ref{homogDLCrem}, which validates our continued application of the term \emph{split-quaternionic} to describe the case $\varepsilon=1$ of Definition \ref{Lcondef} \ref{LconSR} and \emph{quaternionic} for $\varepsilon=-1$.

\vspace{\baselineskip}

\subsection{Adapted Coframings}\label{coframingssec} 

We now construct a bundle whose sections are local coframings adapted to an L-contact structure $\mathbb{L}\to U$ as in Definition \ref{Lcondef}. At first we consider a local $\mathfrak{L}$-Lagrangian splitting $\C\Delta=\Lambda\oplus\overline{\Lambda}$ with $(\overline{\Lambda}_u,\Lambda_u)\in\mathbb{L}_u$ for each $u\in U$ -- i.e., a local section $U\to\mathbb{L}$ -- though the ambiguity associated with such a selection will eventually be incorporated into the construction. With a Lagrangian splitting we get an almost-complex structure on $\Delta$,
\begin{align}\label{LambdaJ}
&J_\Lambda:\Delta\to\Delta
&\text{given by}
&&J_\Lambda|_\Lambda=\im\mathbbm{1},
&&J_\Lambda|_{\overline{\Lambda}}=-\im\mathbbm{1},
\end{align}
and $\C\Delta=\Lambda\oplus\overline{\Lambda}$ is a partially-integrable CR structure on $U$.

Let $\beta:\B_\Lambda\to U$ be the bundle whose fiber over $u\in U$ consists of $\R$-linear isomorphisms adapted to the almost-complex structure \eqref{LambdaJ},
\begin{equation*}
\B_{\Lambda_u}=\beta^{-1}(u)=\left\{\varphi:T_uU\stackrel{\cong}{\longrightarrow} \mathbb{R}\oplus\C^n\ \left|\ 
\varphi(\Delta_u)=\C^n,\ \varphi\circ J_\Lambda=\im\varphi
\right.\right\}.
\end{equation*}
Writing $\R\oplus\C^n$ as column vectors, a section $\theta:U\to\B_\Lambda$ is a local coframing
\begin{align*}
&\begin{bmatrix}\theta^0\\\theta^i\end{bmatrix}\in\Omega^1\left(U,\begin{smallmatrix}\R\\\oplus\\\C^n\end{smallmatrix}\right),
&\theta^0\in\Gamma(\Delta^\bot),
&&\theta^i\in\Gamma(\overline{\Lambda}^\bot) \quad 1\leq i\leq n,
\end{align*}
where $\Delta^\bot\subset T^*U$ and $\overline{\Lambda}^\bot\subset\C T^*U$ denote the annihilators of $\Delta,\overline{\Lambda}$. In particular, $\theta^0$ is a contact form which determines a local matrix representation of the (symplectic) Levi form,
\begin{align}\label{Lmatrep}
&\dd\theta^0\equiv\im\ell_{ij}\theta^i\wedge\overline{\theta}^j\mod\theta^0
&\leadsto[\mathfrak{L}]=\begin{bmatrix}0&\ell_{ij}\\-\ell_{ij}&0\end{bmatrix},
&&\ell_{ij}=\overline{\ell}_{ji}\in C^\infty(\B_\Lambda,\C).
\end{align}
We reduce to the subbundle $\B_\Lambda^1\subset\B_\Lambda$ of \emph{1-adapted} coframes that ``block-diagonalize" $\mathfrak{L}$,
\begin{equation}\label{Lmatdiag}
\B^1_{\Lambda_u}=\{\varphi\in\B_{\Lambda_u}\ |\ \ell_{ij}(\varphi)=\epsilon_{ij}\ \text{ as in } \eqref{epsilonmetric}\}.
\end{equation}
This is a principal bundle $\beta:\B_\Lambda^1\to U$ with structure group
\begin{align}\label{firstGstr}
G_1=\left\{g_1=\begin{bmatrix}|c_0|^2&0\\c^i&c_0c^i_j\end{bmatrix}\in GL\left(\begin{smallmatrix}\R\\\oplus\\\C^n\end{smallmatrix}\right)\ \left|
\begin{array}{c}c_0,c^i,c^i_j\in\C \\ c_0\neq 0\\ \epsilon_{ij}c^i_k\overline{c}^j_l=\epsilon_{kl}\end{array}\right.\right\},
\end{align}
whose Lie algebra $\mathfrak{g}_1$ is matrices of the form
\begin{align}\label{firstGLA}
&\mathfrak{g}_1=\begin{bmatrix}\xi_0+\overline{\xi}_0&0\\\xi^i&\xi_0+\xi^i_j\end{bmatrix},
&\epsilon_{ik}\xi^k_j+\epsilon_{kj}\overline{\xi}^k_i=0.
\end{align}

The tautological 1-form $\lambda\in\Omega^1(\B^1_\Lambda,\R\oplus\C^n)$ is 
\begin{align}\label{lambdataut}
\lambda|_\varphi=\varphi\circ\beta_*.
\end{align}
A section $U\to\B_\Lambda^1$ given by $\theta\in\Omega^1(U,\R\oplus\C^n)$ determines a local trivialization $\B_\Lambda^1\approx G_1\times U$ with respect to which $\lambda$ can be written
\begin{align}\label{lambdagtheta}
\lambda=g_1^{-1}\beta^*\theta,
\end{align}
left-matrix-multiplication by the inverse of $g\in C^\infty(U,G_1)$ in \eqref{firstGstr} giving $\B_\Lambda^1$ a \emph{right}-principal-$G_1$ action. The $\C$-valued parameters \eqref{firstGstr} now serve as fiber coordinates on $\B_\Lambda^1$, and \eqref{lambdagtheta} reads 

\begin{align}\label{localtaut}
\begin{bmatrix}\lambda^0\\\lambda^i\end{bmatrix}=\begin{bmatrix}|c_0|^2&0\\c^i&c_0c^i_j\end{bmatrix}^{-1}\beta^*\begin{bmatrix}\theta^0\\\theta^j\end{bmatrix}.
\end{align}
Differentiating the local expression \eqref{lambdagtheta} yields 
\begin{align}\label{dlambdagtheta}
\dd\lambda&=-g_1^{-1}\dd g_1\wedge\lambda+g_1^{-1}\beta^*\dd\theta,
\end{align}
with $-g_1^{-1}\dd g_1\in\Omega^1(\B_\Lambda^1,\mathfrak{g}_1)$ completing $\lambda$ to a local coframing of $\B^1_\Lambda$ in decidedly \emph{non}-canonical fashion. To wit, with \eqref{firstGLA} and \eqref{localtaut} we rewrite \eqref{dlambdagtheta},
\begin{align}\label{firstSE}
\dd\begin{bmatrix}\lambda^0\\\lambda^i\end{bmatrix}&=-\begin{bmatrix}\xi_0+\overline{\xi}_0&0\\\xi^i&\xi_0+\xi^i_j\end{bmatrix}\wedge\begin{bmatrix}\lambda^0\\\lambda^j\end{bmatrix}+
\begin{bmatrix}\im\epsilon_{ij}\lambda^i\wedge\overline{\lambda}^j\\M^i_{jk}\lambda^j\wedge\overline{\lambda}^k+N^i_{jk}\overline{\lambda}^j\wedge\overline{\lambda}^k\end{bmatrix},
\end{align} 
and having absorbed what torsion we can into $\xi_0,\xi^i,\xi^i_j\in\Omega^1(\B_\Lambda^1,\C)$, these 1-forms are still only determined by the structure equations \eqref{firstSE} up to a transformation of the form 
\begin{align}\label{lambdaprolong}
&\begin{bmatrix}\xi_0\\\xi^i\end{bmatrix}\mapsto\begin{bmatrix}\xi_0\\\xi^i\end{bmatrix}+
\begin{bmatrix}s_0&0\\s^i&s_0\end{bmatrix}
\begin{bmatrix}\lambda^0\\\lambda^i\end{bmatrix},
&s_0,s^i\in C^\infty(\B_\Lambda^1,\C).
\end{align}

The $G_1$-structure equations \eqref{firstSE} are exactly those of the partially integrable CR structure $\C\Delta=\Lambda\oplus\overline{\Lambda}$, with the (harmonic) torsion coefficients $N^i_{jk}$ representing the Nijenhuis tensor of $\Lambda$ (see Remark \ref{Nijenhuisrem}). By pulling back an adapted coframing of $U$ along the bundle projection $\mathbb{L}\to U$, we generalize the structure equations \eqref{firstSE} to those of a general Lagrangian splitting of $\C\Delta$ in $\mathbb{L}$; i.e., an arbitrary section $U\to\mathbb{L}$. Here the equivalence problem branches based on the value of $\varepsilon=\pm1$ in Definition \ref{Lcondef} \ref{LconSR}, and is continued in \S\ref{SQSEsec} for the split-quaternionic case and \S\ref{QSEsec} for the quaternionic case. In both cases, we first pull back adapted coframings along $\mathsf{LaGr}(\C\Delta)\to U$ to an arbitrary Lagrangian splitting satisfying condition \ref{LconSR} of Definition \ref{Lcondef}, and then reduce to $\mathbb{L}$.

\vspace{\baselineskip}

\subsection{Split-Quaternionic Structure Equations}\label{SQSEsec}

Over the bundle $\beta:\B^1_\Lambda\to U$ of 1-adapted coframings of $U$, we construct $\hat{\beta}:\hat{\B}^1\to\B^1_\Lambda$ with fibers
\begin{align}\label{SQLGbundle}
\hat{\beta}^{-1}(\varphi)=\left\{\hat{\varphi}:T_\varphi\B_\Lambda^1\to\R\oplus\C^n\ \left|\  
\hat{\varphi}=\left[\begin{smallmatrix}1&0\\0&\overline{a}_0\mathbbm{1}\end{smallmatrix}\right]\varphi\circ\beta_*-\left[\begin{smallmatrix}0&0\\0&A\end{smallmatrix}\right]\overline{\varphi}\circ\beta_*\left\{ 
\begin{array}{c}
A\in\textsc{Hom}(\C^n,\C^n),\\ 
A^t=A,\\
A\overline{A}=r^2\mathbbm{1}, \\
a_0=1+\im r, r\in\R .
\end{array}
\right.\right.\right\}
\end{align}
By definition, $\beta_*\ker\hat{\varphi}|_{\C T_\varphi\B_\Lambda^1}\subset\C\Delta$ is a Lagrangian subspace framed by an $\mathfrak{L}$-``orthonormal" basis (that is, a basis so that $\mathfrak{L}$ is represented as in \eqref{Lmatrep}, \eqref{Lmatdiag}), with trivial fiber coordinates $A=0$, $r=0$ corresponding to $\hat{\varphi}=\varphi\circ\beta_*$ annihilating the original $\overline{\Lambda}\subset\C\Delta$.  

The tautological 1-form $\eta\in\Omega^1(\hat{\B}_\Lambda^1,\R\oplus\C^n)$ is
\begin{align*}
\eta|_{\hat{\varphi}}=\hat{\varphi}\circ(\hat{\beta})_*,
\end{align*}
or in terms of the tautological forms on $\B^1_\Lambda$,
\begin{align}\label{etatautlambda}
&\eta=\begin{bmatrix}\eta^0\\\eta^i\end{bmatrix},
&\eta^0=(\hat{\beta})^*\lambda^0,
&&\eta^i=(1-\im r)(\hat{\beta})^*\lambda^i-a^i_j(\hat{\beta})^*\overline{\lambda}^j,
\end{align}
for fiber coordinates
\begin{align*}
&r\in C^\infty(\hat{\B}^1), 
&a^i_j=a^j_i\in C^\infty(\hat{\B}^1,\C),
&&\epsilon_{ij}a^i_k\overline{a}^j_l=r^2\epsilon_{kl}.
\end{align*}
In particular, 
\begin{equation*}
(\hat{\beta})^*\lambda^i=(1+\im r)\eta^i+a^i_j\overline{\eta}^j,
\end{equation*}
which we plug into \eqref{firstSE} to obtain
\begin{align*}
\dd\eta^0&=-(\xi_0+\overline{\xi}_0)\wedge\eta^0+\im\epsilon_{ij}\eta^i\wedge\overline{\eta}^j,\\
(\hat{\beta})^*\dd\lambda^i&=-\xi^i\wedge\eta^0-a_0(\delta^i_k\xi_0+\xi^i_k)\wedge\eta^k-(a^i_k\xi_0+a^l_k\xi^i_l)\wedge\overline{\eta}^k\\
&+(a_0M^i_{kn}\overline{a}^n_l+N^i_{mn}\overline{a}^m_k\overline{a}^n_l)\eta^k\wedge\eta^l+(\overline{a}_0M^i_{ml}a^m_k+\overline{a}_0^2N^i_{kl})\overline{\eta}^k\wedge\overline{\eta}^l\\
&+(|a_0|^2M^i_{kl}-M^i_{mn}a^m_l\overline{a}^n_k+2\overline{a}_0N^i_{ml}\overline{a}^m_k)\eta^k\wedge\overline{\eta}^l,
\end{align*}
where we have recycled the names of the pseudoconnection forms $\xi_0=(\hat{\beta})^*\xi_0,\xi^i=(\hat{\beta})^*\xi^i,\xi^i_j=(\hat{\beta})^*\xi^i_j$. These give way to
\begin{equation}\label{splitLGSE}
\begin{aligned}
\dd\eta^0&=(\varsigma+\overline{\varsigma})\wedge\eta^0+\im\epsilon_{ij}\eta^i\wedge\overline{\eta}^j,\\
\dd\eta^i&=\zeta^i\wedge\eta^0+\varsigma\wedge\eta^i+(\delta^i_jr\dd r-\overline{a}^k_j\dd a^i_k)\wedge\eta^j-\zeta^i_j\wedge\eta^j+\kappa^i_j\wedge\overline{\eta}^j\\
&+O^i_{kl}\eta^k\wedge\eta^l+P^i_{kl}\eta^k\wedge\overline{\eta}^l+Q^i_{kl}\overline{\eta}^k\wedge\overline{\eta}^l,
\end{aligned}
\end{equation}
for 1-forms
\begin{equation}\label{SLGSEforms}
\begin{aligned}
&\varsigma=-(\xi_0+\im\dd r+r^2(\xi_0-\overline{\xi}_0)),
&&\zeta^i_j=(1+r^2)\xi^i_j-a^i_k\overline{a}^l_j\overline{\xi}^k_l,\\
&\zeta^i=-((1-\im r)\xi^i-a^i_j\overline{\xi}^j),
&&\kappa^i_j=-\im a^i_j\dd r-(1-\im r)(\dd a^i_j+a^i_j(\xi_0-\overline{\xi}_0)+(a^k_j\xi^i_k-a^i_k\overline{\xi}^k_j)),
\end{aligned}
\end{equation}
and torsion coefficients
\begin{equation}\label{SLGSEtorsion}
\begin{aligned}
O^i_{kl}&=|a_0|^2M^i_{kn}\overline{a}^n_l+\overline{a}_0N^i_{mn}\overline{a}^m_k\overline{a}^n_l-a_0a^i_j(\overline{M}^j_{ml}\overline{a}^m_k+a_0\overline{N}^j_{kl}),\\
P^i_{kl}&=\overline{a}_0(|a_0|^2M^i_{kl}-M^i_{mn}a^m_l\overline{a}^n_k+2\overline{a}_0N^i_{ml}\overline{a}^m_k)+a^i_j(|a_0|^2\overline{M}^j_{lk}-\overline{M}^j_{mn}\overline{a}^m_ka^n_l+2a_0\overline{N}^j_{mk}a^m_l),\\
Q^i_{kl}&=\overline{a}_0^2(M^i_{ml}a^m_k+\overline{a}_0N^i_{kl})-a^i_ja^n_l(\overline{a}_0\overline{M}^j_{kn}+\overline{N}^j_{mn}a^m_k).
\end{aligned}
\end{equation}

To reduce from general coframes of $\mathsf{LaGr}(\C\Delta)$ to those of $\mathbb{L}$, we construct a bundle $\B^2\subset\hat{\B}^1$ whose coframes \eqref{SQLGbundle} annihilate $\overline{\Lambda}_A=\beta_*\ker\hat{\varphi}\subset\C\Delta$ such that $(\overline{\Lambda}_A,\Lambda_A)\in\mathbb{L}$. Condition \ref{Lconcurve} of Definition \ref{Lcondef} states that the fiber coordinates $A=a^i_j$ of $\B^2$ will be functions of a single complex variable $a\in C^\infty(\B^2,\C)$. In fact, we can bring the coordinates into normal form (\cite[Rem. 1.3]{PZ}) and reduce the structure group of $\B^2$ even further: $\B^2$ is not a bundle over $\B^1_\Lambda$, but over $\B^2_\Lambda\subset\B^1_\Lambda$ whose coframes are adapted such that $\hat{\beta}:\B^2\to\B^2_\Lambda$ has fibers \eqref{SQLGbundle} with $A$ diagonalized (\cite[Cor. 4.6(1)]{PZ}),  
\begin{align}\label{Adiag}
&a^i_j=a\delta^i_j,
&a\in C^\infty(\B^2,\C),
&&r=|a|.
\end{align}
Revisiting the language of \S\ref{2nonsec}, for L-contact structures that arise as leaf spaces of 2-nondegenerate CR manifolds, the fiber coordinates $A=a^i_j$ encode the maps $\mathrm{ad}_K:\overline{\Lambda}\to\Lambda$, which are represented in the structure equations by the matrix $\kappa^i_j$ \eqref{SLGSEforms} (cf. \cite[\S2.4, \S3.2]{CPCAG}). The Hermitian version of our Levi form $\mathfrak{L}$ is obtained by complex-conjugating the second argument, so that the pair $(\mathfrak{L},A)$ can be brought into normal form by choosing appropriate bases of $(\overline{\Lambda},\Lambda)\in\mathbb{L}$ (\cite[\S4 esp. Def. 4.3]{PZ}). Even in the generalized setting of Definition \ref{Lcondef} (i.e., an L-contact structure not necessarily given by a leaf space), condition \ref{LconSR} ensures that the pair $(\mathfrak{L},A)$ is \emph{strongly non-nilpotent regular} so that its normal form specializes from the description in \cite[Thm. 4.4]{PZ} to that of \cite[Cor. 4.6(1)]{PZ}. With $A$ diagonalized, 
\begin{align*}
&\kappa^i_j=\delta^i_j\kappa\in\Omega^1(\B^2,\C)
&(\text{compare to }\eqref{splithomog2non}).
\end{align*} 
In particular, after reducing $\B^2_\Lambda\subset\B^1_\Lambda$, the forms $\xi^i_j\in\Omega^1(\B_\Lambda^1,\mathfrak{su}(p,q))$ as in \eqref{firstGLA} are no longer independent, but satisfy
\begin{align}\label{xireal}
a^k_j\xi^i_k-a^i_k\overline{\xi}^k_j=a(\xi^i_j-\overline{\xi}^i_j)\equiv0\mod\{\lambda^0,\lambda^i,\overline{\lambda}^j\},
\end{align}
and we are left with 
\begin{align*}
&\gamma^i_j=\Re\xi^i_j\in\Omega^1(\B^2_\Lambda,\mathfrak{so}(p,q))
&(\text{compare to }\eqref{splitLAmat}),
\end{align*}
along with the expansion of \eqref{xireal},
\begin{align}\label{xireduceR}
&\xi^i_j-\overline{\xi}^i_j=\im (R^i_{j0}\lambda^0+R^i_{jk}\lambda^k+\overline{R}^i_{jk}\overline{\lambda}^k)
&\text{for some}
&&R^i_{j0}\in C^\infty(\B^2_\Lambda),\ R^i_{jk}\in C^\infty(\B^2_\Lambda,\C).
\end{align}

As before, we keep the same names of the pseudoconnection forms $\xi_0=(\hat{\beta})^*\xi_0,\xi^i=(\hat{\beta})^*\xi^i,\gamma^i_j=(\hat{\beta})^*\gamma^i_j$ and torsion coefficients $M,N$ when we pull back along $\hat{\beta}:\B^2\to\B^2_\Lambda$ to get 
\begin{equation}\label{splitLGSE2}
\begin{aligned}
\dd\eta^0&=(\varsigma+\overline{\varsigma})\wedge\eta^0+\im\epsilon_{ij}\eta^i\wedge\overline{\eta}^j,\\
\dd\eta^i&=\zeta^i\wedge\eta^0+\varsigma\wedge\eta^i-\gamma^i_j\wedge\eta^j+\kappa\wedge\overline{\eta}^i+O^i_{kl}\eta^k\wedge\eta^l+P^i_{kl}\eta^k\wedge\overline{\eta}^l+Q^i_{kl}\overline{\eta}^k\wedge\overline{\eta}^l,
\end{aligned}
\end{equation}
for 1-forms
\begin{align}\label{SE2forms}
&\varsigma=-(\xi_0+\im\dd r+r^2(\xi_0-\overline{\xi}_0)+\tfrac{1}{2}(\overline{a}\dd a-a\dd\overline{a})),
&\zeta^i=-((1-\im r)\xi^i-a\overline{\xi}^i+\im R^i_{j0}(\tfrac{1}{2}(1+2r^2)\eta^j+a\overline{\eta}^j)),
\end{align}
and torsion coefficients
\begin{equation}\label{SE2torsion}
\begin{aligned}
O^i_{kl}&=-\tfrac{\im}{2}(1+2r^2)(a_0R^i_{kl}+\overline{a}\overline{R}^i_{kl})+
\overline{a}|a_0|^2M^i_{kl}+\overline{a}_0\overline{a}^2N^i_{kl}-a_0r^2\overline{M}^i_{kl}-aa_0^2\overline{N}^i_{kl},\\
P^i_{kl}&=-\tfrac{\im}{2}(1+2r^2)(\overline{a}_0\overline{R}^i_{kl}+aR^i_{kl})+\overline{a}_0|a_0|^2M^i_{kl}-\overline{a}_0r^2M^i_{lk}+2\overline{a}\overline{a}_0^2N^i_{kl}\\
&+a|a_0|^2\overline{M}^i_{lk}-ar^2\overline{M}^i_{kl}+2a_0a^2\overline{N}^i_{lk}-\im a(a_0R^i_{kl}+\overline{a}\overline{R}^i_{kl}),\\
Q^i_{kl}&=a\overline{a}_0^2M^i_{kl}+\overline{a}_0^3N^i_{kl}-\overline{a}_0a^2\overline{M}^i_{kl}-a^3\overline{N}^i_{kl}
-\im a(\overline{a}_0\overline{R}^i_{kl}+aR^i_{kl}).
\end{aligned}
\end{equation}

As in \S\ref{coframingssec} (see \eqref{lambdaprolong}), the pseudoconnection 1-forms $\varsigma,\zeta^i,\kappa\in\Omega^1(\B^2,\C)$ are not uniquely determined by the structure equations \eqref{splitLGSE2}, but only up to a substitution
\begin{align}\label{etaprolong}
&\begin{bmatrix}\varsigma\\\zeta^i\\\kappa\end{bmatrix}\mapsto\begin{bmatrix}\varsigma\\\zeta^i\\\kappa\end{bmatrix}+
\begin{bmatrix}s_0&0&0\\s^i&s_0&s_1\\s_1&0&0\end{bmatrix}
\begin{bmatrix}\eta^0\\\eta^i\\\overline{\eta}^i\end{bmatrix},
&s_0,s^i,s_1\in C^\infty(\B^2,\C).
\end{align}
To complete the construction of an absolute parallelism over $U$ via Cartan's method of equivalence, one would pull back the structure equations \eqref{splitLGSE2} to the bundle of all such pseudoconnection forms \eqref{etaprolong} over $\B^2$ and differentiate them with the help of the identity $\dd^2=0$ to determine the exterior derivatives of the tautological forms on the prolonged bundle (\cite[\S3.5]{CPCAG}). Then, normalizing torsion in the new structure equations (\cite[\S\S 3.4-3.5]{CPCAG}) would eventually reduce the fiber coordinates \eqref{etaprolong} to a single, $\R$-valued function $s_0$ whose corresponding 1-form in the structure equations \eqref{splitMCeq} of the homogeneous model was denoted $\zeta^0$. 

\begin{rem}\label{Lcontorrem}
We have no need to continue with the prolongation procedure, which was shown to terminate as expected in \cite{PZ} in the 2-nondegenerate CR case (and this carries over to the general L-contact case by \cite{SykesZelenko}). Indeed, the structure equations \eqref{splitLGSE2} are already in their final form, and are sufficient to compare to the homogeneous model \eqref{splitMCeq} for our purposes. Namely, in order that \eqref{splitLGSE2} describes an L-contact structure which is locally equivalent to the homogeneous model, it is necessary that the torsion coefficients \eqref{SE2torsion} are zero.
\end{rem}

\vspace{\baselineskip}

\subsection{Quaternionic Structure Equations}\label{QSEsec} 

We carry out the same sequence of constructions as in \S\ref{SQSEsec}, adjusting the structure equations as needed. To begin, we note that the Levi form \eqref{Lmatrep}, \eqref{Lmatdiag} on $\B^1_\Lambda$ is normalized to
\begin{align}\label{Qdlambda0}
&\dd\lambda^0=-(\xi_0+\overline{\xi}_0)\wedge\lambda^0+\im\delta_{ij}\lambda^i\wedge\overline{\lambda}^j-\im\delta_{ij}\lambda^{p+i}\wedge\lambda^{p+j}
&(\text{cf. \S\ref{Qlinessec}}),
\end{align} 
and in this section indices range from 1 to $p=\tfrac{1}{2}n$. The bundle $\hat{\beta}:\hat{\B}^1\to\B^1_\Lambda$ has fibers 
\begin{align*}
\hat{\beta}^{-1}(\varphi)=\left\{\hat{\varphi}:T_\varphi\B_\Lambda^1\to\R\oplus\C^n\ \left|\  
(1+r^2)\hat{\varphi}=\varphi\circ\beta_*
-\left[\begin{smallmatrix}0&0&0\\0&0&-A\\0&A&0\end{smallmatrix}\right]\overline{\varphi}\circ\beta_*\left\{ 
\begin{array}{c}
A\in\textsc{Hom}(\C^p,\C^p),\\ 
A^t=A,\\
A\overline{A}=r^2\mathbbm{1}, r\in\R .
\end{array}
\right.\right.\right\},
\end{align*}
and tautological forms
\begin{align}\label{QSEeta}
&\eta^0=\tfrac{1}{1+r^2}\hat{\beta}^*\lambda^0,
&\eta^i=\tfrac{1}{1+r^2}(\hat{\beta}^*\lambda^i+a^i_j\hat{\beta}^*\overline{\lambda}^{p+j}),
&&\eta^{p+i}=\tfrac{1}{1+r^2}(\hat{\beta}^*\lambda^{p+i}-a^i_j\hat{\beta}^*\overline{\lambda}^{j}),
\end{align}
for fiber coordinates
\begin{align*}
&r\in C^\infty(\hat{\B}^1), 
&a^i_j=a^j_i\in C^\infty(\hat{\B}^1,\C),
&&\delta_{ij}a^i_k\overline{a}^j_l=r^2\delta_{jl}.
\end{align*}
In particular,
\begin{align*}
&\hat{\beta}^*\lambda^0=(1+r^2)\eta^0,
&\hat{\beta}^*\lambda^i=\eta^i-a^i_j\overline{\eta}^{p+j},
&&\hat{\beta}^*\lambda^{p+i}=\eta^{p+i}+a^i_j\overline{\eta}^{j},
\end{align*}
and we pull back \eqref{firstSE} to get
\begin{align*}
\dd\eta^0&=(\varsigma+\overline{\varsigma})\wedge\eta^0+\im\delta_{ij}\eta^i\wedge\overline{\eta}^j-\im\delta_{ij}\eta^{p+i}\wedge\eta^{p+j},\\
\dd\eta^i&\equiv\zeta^i\wedge\eta^0+\varsigma\wedge\eta^i+\tfrac{1}{2(1+r^2)}(\overline{a}^k_j\dd a^i_k-a^i_k\dd\overline{a}^k_j)\wedge\eta^j\\
&-\zeta^i_j\wedge\eta^j-\zeta^i_{p+j}\wedge\eta^{p+j}+\kappa^i_j\wedge\overline{\eta}^j+\kappa^i_{p+j}\wedge\overline{\eta}^{p+j}
&\mod\{\eta\wedge\eta,\ \eta\wedge\overline{\eta},\ \overline{\eta}\wedge\overline{\eta}\},\\
\dd\eta^{p+i}&\equiv\zeta^{p+i}\wedge\eta^0+\varsigma\wedge\eta^{p+i}+\tfrac{1}{2(1+r^2)}(\overline{a}^k_j\dd a^i_k-a^i_k\dd\overline{a}^k_j)\wedge\eta^{p+j}\\
&-\zeta^{p+i}_j\wedge\eta^j-\zeta^{p+i}_{p+j}\wedge\eta^{p+j}+\kappa^{p+i}_j\wedge\overline{\eta}^j+\kappa^{p+i}_{p+j}\wedge\overline{\eta}^{p+j}
&\mod\{\eta\wedge\eta,\ \eta\wedge\overline{\eta},\ \overline{\eta}\wedge\overline{\eta}\},
\end{align*}
(having reduced modulo torsion terms), where
\begin{align*}
&\varsigma=-\left(r\dd r+\tfrac{1}{1+r^2}(\xi_0+r^2\overline{\xi}_0)\right),
&\zeta^i=-(\xi^i+a^i_j\overline{\xi}^{p+j}),
&&\zeta^{p+i}=-(\xi^{p+i}-a^i_j\overline{\xi}^j),
\end{align*}
$\xi^i_j\in\Omega^1(\B^1_\Lambda\mathfrak{su}(p,p))$ have pulled back to
\begin{equation*}
\begin{bmatrix}\zeta^i_j&\zeta^i_{p+j}\\\zeta^{p+i}_j&\zeta^{p+i}_{p_j}\end{bmatrix}=
\frac{1}{1+r^2}\begin{bmatrix}
\xi^i_j+a^i_k\overline{a}^l_j\overline{\xi}^{p+k}_{p+l}&
\xi^i_{p+j}-a^i_k\overline{a}^l_j\overline{\xi}^{p+k}_l\\
\xi^{p+i}_j-a^i_k\overline{a}^l_j\overline{\xi}^k_{p+l}&
\xi^{p+i}_{p+k}+a^i_j\overline{a}^l_k\overline{\xi}^j_l
\end{bmatrix},
\end{equation*}
and the ``$\mathrm{ad}_K$" maps are represented
\begin{align}\label{Qkappa}
\begin{bmatrix}\kappa^i_j&\kappa^i_{p+j}\\\kappa^{p+i}_j&\kappa^{p+i}_{p+j}\end{bmatrix}
=
\frac{1}{1+r^2}
\begin{bmatrix}
-a^l_k\xi^i_{p+l}-a^i_j\overline{\xi}^{p+j}_k&
\dd a^i_j+a^i_k(\xi_0-\overline{\xi}_0)+a^l_k\xi^i_l-a^i_j\overline{\xi}^{p+j}_{p+k}\\
-\dd a^i_j-a^i_k(\xi_0-\overline{\xi}_0)-a^l_k\xi^{p+i}_{p+l}+a^i_j\overline{\xi}^j_k&
a^l_k\xi^{p+i}_l+a^i_j\overline{\xi}^j_{p+k}
\end{bmatrix}.
\end{align}

As in \S\ref{SQSEsec}, we invoke \cite[Cor. 4.6(2)]{PZ} to reduce to $\B^2\subset\hat{\B}^1$ over $\B^2_\Lambda\subset\B^1_\Lambda$ consisting of coframes of $\mathbb{L}$ such that the fibers of $\B^2\to\B^2_\Lambda$ have $A$ diagonalized \eqref{Adiag}. With \eqref{Qkappa} normalized as in \eqref{quathomog2non} (modulo torsion terms), the pseudoconnection forms $\xi\in\Omega^1(\B^1_\Lambda,\mathfrak{su}(p,p))$ are no longer independent over $\B^2_\Lambda$, but satisfy \eqref{QLAmat} (modulo tautological forms). Modulo torsion terms, the final structure equations on $\B^2$ are equivalent to \eqref{quatMCeq}, with 
\begin{align*}
&\varsigma=-\left(r\dd r+\tfrac{1}{1+r^2}(\xi_0+r^2\overline{\xi}_0)\right)+\tfrac{1}{2(1+r^2)}(\overline{a}\dd a-a\dd\overline{a}).
\end{align*}
Here again, the pseudoconnection forms are not uniquely determined by the structure equations, but require prolongation before they may be canonically defined by higher order structure equations (see the discussion at the end of \S\ref{SQSEsec}).

\vspace{\baselineskip}


\section{L-contact Structure on Unit Tangent Bundles}\label{UMLcon}


\subsection{The Unit Tangent Bundle of a Semi-Riemannian Manifold}\label{UTBsec}


Let $M$ be a smooth manifold and $\mu:TM\to M$ its tangent bundle; $\mu(y)=x$ $\forall y\in T_xM$. The pushforward $\mu_*:TTM\to TM$ annihilates the vertical bundle $T^{\uparrow}M=\ker\mu_*$, which is naturally identified with $TM$ as follows. For $f\in C^\infty(M)$, $\dd f\in\Omega^1(M)\subset C^\infty(TM)$ is regarded as a function on $TM$, linear on the fibers, and each $Y\in\Gamma(TM)$ determines $Y^\uparrow\in\Gamma(T^\uparrow M)$ by their actions as derivations,
\begin{align*}
Y(f)=Y^\uparrow(\dd f).
\end{align*} 

Now suppose $M$ is equipped with an affine connection $\nabla:\Gamma(TM)\to\Gamma(T^*M\otimes TM)$. Each $y\in T_xM$ can be locally extended to $Y\in\Gamma(TM)$ so that $\nabla_XY=0$ $\forall X\in\Gamma(TM)$, which uniquely determines the 1-jet of $Y:M\to TM$ at $x\in M$. The pushforward $Y_*:TM\to TTM$, satisfies $\mu_*\circ Y_*=\mathbbm{1}$, hence $Y_*(T_xM)\subset T_yT_xM$ defines the horizontal subbundle $\overrightarrow{T}M\subset TTM$ which is complementary to $T^\uparrow M$ and canonically isomorphic to $TM$ via $\mu_*$; i.e., each $Y\in\Gamma(TM)$ determines $\overrightarrow{Y}\in\Gamma(\overrightarrow{T}M)$ (\cite[\S1.3.1]{CapSlovak}). The endomorphism field $J:\overrightarrow{T}M\oplus T^\uparrow M\to \overrightarrow{T}M\oplus T^\uparrow M$ on $TM$ will satisfy $J^2=-\mathbbm{1}$ if we set
\begin{align}\label{TMaCstr}
&J(\overrightarrow{Y})=Y^\uparrow,
&J(Y^\uparrow)=-\overrightarrow{Y},
&&\forall Y\in TM.
\end{align}
Similarly, $K:TTM\to TTM$ with $K^2=\mathbbm{1}$ and $JK=-KJ$ is given by 
\begin{align}\label{TMaSQstr}
&K(\overrightarrow{Y})=Y^\uparrow,
&K(Y^\uparrow)=\overrightarrow{Y},
&&\forall Y\in TM.
\end{align}

Now suppose $(M,g)$ is semi-Riemannian and $\nabla$ is the Levi-Civita connection of $g\in S^2T^*M$. The Sasaki metric $\hat{g}\in S^2T^*TM$ is defined by
\begin{align*}
&\hat{g}(\overrightarrow{X},\overrightarrow{Y})=\hat{g}(X^\uparrow,Y^\uparrow)=g(X,Y),
&\hat{g}(\overrightarrow{X},Y^\uparrow)=0,
&&\forall X,Y\in TM,
\end{align*}
from which it follows immediately
\begin{align}\label{gJ}
&\hat{g}(J\hat{X},J\hat{Y})=\hat{g}(\hat{X},\hat{Y})
&\Rightarrow \hat{g}(J\hat{X},\hat{Y})+\hat{g}(\hat{X},J\hat{Y})=0
&&\forall\hat{X},\hat{Y}\in TTM.
\end{align}
On the other hand,
\begin{align}\label{gK}
&\hat{g}(K\hat{X},\hat{Y})=\hat{g}(\hat{X},K\hat{Y})
&\Rightarrow \hat{g}(K\hat{X},J\hat{Y})+\hat{g}(\hat{X},JK\hat{Y})=0
&&\forall\hat{X},\hat{Y}\in TTM.
\end{align}
We depict the triple $\hat{g},J,K$ schematically as
\begin{align}\label{hatgJKmat}
&\hat{g}=\begin{bmatrix}g&0\\0&g\end{bmatrix},
&J=\begin{bmatrix}0&-\mathbbm{1}\\\mathbbm{1}&0\end{bmatrix},
&&K=\begin{bmatrix}0&\mathbbm{1}\\\mathbbm{1}&0\end{bmatrix}
&&\text{on  }TTM=\begin{array}{c}\overrightarrow{T}M\\\oplus\\T^\uparrow M\end{array}.
\end{align}

If we suppose further that $g$ has at least one positive eigenvalue, we can consider the unit tangent bundle \eqref{UMdef} of $M$. For each $u\in U_xM$, $u^\bot=\ker g(u,\cdot)\subset T_xM$ is a hyperplane implicitly identified with $T_u(U_xM)$ (as one would for a sphere in Euclidean space). Explicitly, $T_uU_xM=\overrightarrow{T_xM}\oplus (u^\bot)^\uparrow\subset \overrightarrow{T}M\oplus T^\uparrow M$ and we name the corank-1 distribution of $TUM$,
\begin{equation}\label{Deltasplitdef}
\Delta_u=\overrightarrow{(u^\bot)}\oplus(u^\bot)^\uparrow\subset T_uUM.
\end{equation}
Equivalently, $\Delta=\ker\theta$ for the contact form $\theta\in\Omega^1(UM)$ given by
\begin{align}\label{thetadef}
\theta|_u=\hat{g}(\overrightarrow{u},\cdot).
\end{align}
The endomorphism field $J$ on $TM$ restricts to an almost-complex structure on $\Delta$, whose complexification splits 
\begin{equation}\label{UMLambdasplit}
\C\Delta=\Lambda\oplus\overline{\Lambda}
\end{equation}
into $\im$ and $-\im$ eigenspaces of $J$, respectively. 

\begin{rem}\label{Jintrem}
The splitting \eqref{UMLambdasplit} is the \emph{standard CR structure} of $UM$ (\cite{Tanno}), in spite of the fact that it's rarely CR-integrable (see Remark \ref{Nijenhuisrem}). Indeed, the almost-complex structure \eqref{TMaCstr} of $TM$ is integrable if and only if $M$ is flat as a semi-Riemannian manifold (\cite{Hsu}\footnote{\cite{Hsu} attributes this to \cite{Nagano}, but I was unable to locate a copy of Nagano's article. Note that these sources only explicitly refer to the Riemannian (definite-signature) case.}, \cite{Dombrowski}), so for $\dim M>2$ the Nijenhuis tensor of $\Lambda$ only vanishes for highly symmetrical $M$, e.g. Riemannian space forms of sectional curvature 1 (\cite{BarDrag}).
\end{rem}

By definition \eqref{TMaCstr},
\begin{align}\label{UMLambda}
&\Lambda_u=\C\left\{\left.\overrightarrow{Y}-\im Y^\uparrow\ \right|\ Y\in u^\bot\right\},
&\overline{\Lambda}_u=\C\left\{\left.\overrightarrow{Y}+\im Y^\uparrow\ \right|\ Y\in u^\bot\right\},
\end{align}
and \eqref{hatgJKmat} restricts to 
\begin{align}\label{hatgJKmatLambda}
&\hat{g}=\begin{bmatrix}0&2g|_{u^\bot}\\2g|_{u^\bot}&0\end{bmatrix},
&J=\begin{bmatrix}\im\mathbbm{1}&0\\0&-\im\mathbbm{1}\end{bmatrix},
&&K=\begin{bmatrix}0&\im\mathbbm{1}\\-\im\mathbbm{1}&0\end{bmatrix}
&&\text{on  }\C\Delta=\begin{array}{c}\Lambda_u\\\oplus\\\overline{\Lambda}_u\end{array}.
\end{align}

The Levi form of $UM$ is
\begin{align*}
&\mathfrak{L}(\hat{Y}_1,\hat{Y}_2)=-\im\dd\theta(\hat{Y}_1,\hat{Y}_2)
&\hat{Y}_1,\hat{Y}_2\in\Gamma(\C\Delta).
\end{align*}

\begin{lem}\label{LeviSasakilem} 
$\mathfrak{L}(\hat{Y}_1,\hat{Y}_2)=\tfrac{1}{2}\hat{g}(\hat{Y}_1,J\hat{Y}_2)$\hspace{1cm}  $\forall \hat{Y}_1,\hat{Y}_2\in\Gamma(\C\Delta)$. 
\end{lem}
\begin{proof}
This follows from \eqref{gJ}, and will be shown in \S\ref{ONFBsec}. Here we wish to acknowledge the fact that the contact metric on $UM$ is homothetic to the Sasaki metric; see \cite{BlairContact}, keeping in mind that \eqref{thetadef} is equivalent to $\hat{g}(u^\uparrow,J\cdot)$, and the Levi form is only defined up to a choice of local trivialization $\C(TUM/\Delta)\approx\C$, provided here by $\theta$.
\end{proof}

In light of Lemma \ref{LeviSasakilem} and \eqref{gJ}, \eqref{gK}, we offer the following

\begin{definition}\label{stdSQ}
On $UM$ with its contact distribution \eqref{Deltasplitdef}, the endomorphisms $J,K:\Delta\to\Delta$ defined by restriction of \eqref{TMaCstr} and \eqref{TMaSQstr} define the \emph{standard split-quaternionic structure} of the unit tangent bundle of $M$. The corresponding L-contact structure (see Remark \ref{homogDLCrem}) is also called \emph{standard} (cf. Remark \ref{Jintrem}).
\end{definition}

The homogeneous leaf space $U$ presented as split-quaternionic lines in \S\ref{SQlinessec} can be locally realized as the unit tangent bundle $UM$ for $M=\R^{p+1,q}$ (see Appendix \ref{appendixsec}) with its standard L-contact structure. Here, $\R^{p+1,q}$ is semi-Euclidean space with its diagonalized (flat) metric of signature $(p+1,q)$. An L-contact manifold is ``non-flat" if it is not equivalent to the homogeneous model (see Remark \ref{Lcontorrem}), which is measured at lowest order by torsion terms in the structure equations \eqref{splitLGSE2}. Similarly, $M$ is non-flat if it is not locally equivalent to $\R^{p+1,q}$, which is measured by the Riemann curvature tensor of $M$. In order to relate these two notions of curvature (``non-flatness") in \S\ref{ONFBsec}, let us establish some notation. 

The Riemann curvature tensor of $g$ is
\begin{align*}
R\in\Omega^2(M,\mathfrak{so}(TM)),
\end{align*}
so for $X,Y\in T_xM$,
\begin{align}\label{curvop}
&R(X,Y):T_xM\to T_xM
&\text{satisfies}
&&g(R(X,Y)y_1,y_2)+g(y_1,R(X,Y)y_2)=0,
&&\forall y_1,y_2\in T_x M.
\end{align}
In particular, if we set $y_1=y_2=u$, we conclude that $R(\cdot,\cdot)u:\wedge^2(u^\bot)\to u^\bot$ is well-defined, hence we can extend by conjugate-bilinearity and lift to
\begin{align}\label{RNijen}
R(\cdot,\cdot)\overrightarrow{u}:\Lambda_u\wedge\Lambda_u\to\Lambda_u.
\end{align}
Compare \eqref{RNijen} to the Nijenhuis tensor discussed in Remark \ref{Nijenhuisrem}.

Contracting $R$ with $g$ yields the Ricci tensor $\mathrm{Ric}\in S^2T^*M$, which we use to construct the \emph{Ricci-shift operator},
\begin{align}\label{Ricshift}
&\overrightarrow{X}\mapsto \overrightarrow{X}+\mathrm{Ric}(u,X)X^\uparrow,
&X^\uparrow\mapsto X^\uparrow+ \mathrm{Ric}(u,X)\overrightarrow{X}
&&\forall u,X\in T_xM\leadsto \overrightarrow{X},X^\uparrow\in T_u UM.
\end{align}
\begin{definition}\label{RicSQ}
On $UM$ with its contact distribution \eqref{Deltasplitdef}, the \emph{Ricci-shifted L-contact structure} is defined by application of the Ricci shift operator \eqref{Ricshift} to each Lagrangian subspace in the standard L-contact structure (Definition \ref{stdSQ}),
\end{definition}

\noindent Note that if $M$ is Ricci-flat, the Ricci-shifted L-contact structure coincides with the standard one.

\vspace{\baselineskip}

\subsection{The Orthonormal Frame Bundle}\label{ONFBsec}

Let $W=\R^{n+1}$ with its standard basis of column vectors $e_0,\dots,e_n\in W$. Indices written in {\tt typewriter font} range from $0$ to $n$ while those in standard font start counting at 1. The metric $\langle\cdot,\cdot\rangle\in S^2W^*$ of signature $(p+1,q)$ is 
\begin{equation*}
\langle e_\mathtt{i},e_\mathtt{j}\rangle=\epsilon_{\mathtt{ij}}
\text{ as in }\eqref{epsilonmetric}, \text{ adding }\epsilon_0=1.
\end{equation*}
$(M,g)$ is a semi-Riemannian manifold of signature $(p+1,q)$, $p+q=n$. Over each $x\in M$, the orthonormal frame bundle $\pi:\F\to M$ has fiber
\begin{align*}
\pi^{-1}(x)=\{\phi:W\to T_x M\ |\ \langle w_1,w_2\rangle=g(\phi(w_1),\phi(w_2))\ \forall w_1,w_2\in W\},
\end{align*}
which carries a right principal action of the orthogonal group $O(W,\langle\cdot,\cdot\rangle)=O(p+1,q)$ given by composing $\phi$ with orthogonal transformations of $W$. The tautological 1-form $\omega\in\Omega^1(\F,W)$ is
\begin{equation}\label{omegadef}
\omega|_\phi=\phi^{-1}\circ\pi_*,
\end{equation}
and the Levi-Civita connection form $\gamma\in\Omega^1(\F,\mathfrak{so}(W,\langle\cdot,\cdot\rangle))$ is determined by the torsion-free structure equation
\begin{equation}\label{domega}
\dd\omega=-\gamma\wedge\omega.
\end{equation}
With respect to the standard basis of $W$, $\omega=\omega^\mathtt{i}\otimes e_\mathtt{i}$ for $\omega^\mathtt{i}\in\Omega^1(\F)$ so that 
\begin{equation}\label{Fliftg}
\pi^*g=\epsilon_{\mathtt{ij}}\omega^\mathtt{i}\otimes\omega^\mathtt{j},
\end{equation}
and \eqref{domega} becomes
\begin{align*}
&\dd\omega^\mathtt{i}=-\gamma^\mathtt{i}_\mathtt{j}\wedge\omega^\mathtt{j},
&\epsilon_\mathtt{i}\gamma^\mathtt{i}_\mathtt{j}+\epsilon_\mathtt{j}\gamma^\mathtt{j}_\mathtt{i}=0,
\end{align*}
where the latter holds for each fixed $\mathtt{i},\mathtt{j}$ (not summed over), reflecting the fact that $\gamma^\mathtt{i}_\mathtt{j}$ is $\mathfrak{so}(p+1,q)$-valued. It will be convenient to introduce $\tilde{\gamma}^\mathtt{i}_\mathtt{j}$ taking values in $\mathfrak{so}(p+q+1)$ via
\begin{align}\label{domegai}
&\dd\omega^\mathtt{i}=-\epsilon^\mathtt{j}_\mathtt{k}\tilde{\gamma}^\mathtt{i}_\mathtt{j}\wedge\omega^\mathtt{k},
&\epsilon^\mathtt{j}_\mathtt{k}=\epsilon_{\mathtt{jk}},
&&\tilde{\gamma}^\mathtt{i}_\mathtt{j}+\tilde{\gamma}^\mathtt{j}_\mathtt{i}=0.
\end{align}
The rest of the semi-Riemannian structure equations read
\begin{align}\label{dLeviCiv}
\dd\gamma^\mathtt{i}_\mathtt{j}&=-\gamma^\mathtt{i}_\mathtt{k}\wedge\gamma^\mathtt{k}_\mathtt{j}+\tfrac{1}{2}R^\mathtt{i}_{\mathtt{jkl}}\omega^\mathtt{k}\wedge\omega^\mathtt{l}. 
\end{align}

For each 1-form $\omega^\mathtt{i},\gamma^\mathtt{i}_\mathtt{j}\in\Omega^1(\F)$ in the coframing of $\F$, $\partial_{\omega^\mathtt{i}},\partial_{\gamma^\mathtt{i}_\mathtt{j}}\in\Gamma(T\F)$ will denote their dual vector fields. In particular, by definition \eqref{omegadef},
\begin{equation}\label{tautdual}
\omega(\partial_{\omega^\mathtt{i}})=e_\mathtt{i}.
\end{equation}
$\F$ fibers over the unit tangent bundle \eqref{UMdef} via the projection $\pi_0:\F\to UM$ mapping each frame to its first basis vector,
\begin{align*}
\pi_0(\phi)=\phi(e_0)\in U_xM.
\end{align*}
In light of \eqref{tautdual}, $\pi_0(\phi)=\pi_*\partial_{\omega^0}$, and since $\pi=\mu\circ\pi_0$ for the basepoint projection $\mu:TM\to M$, we have $(\pi_0)_*\partial_{\omega^0}=\overrightarrow{u}\in\overrightarrow{T}M$ for $u=\pi_0(\phi)\in UM$. Along with \eqref{Fliftg}, the fact that $\mu^*g|_{\overrightarrow{T}M}=\hat{g}|_{\overrightarrow{T}M}$ then implies that the contact form \eqref{thetadef} pulls back to
\begin{equation}\label{omega0contact}
(\pi_0)^*\theta=\omega^0.
\end{equation}
Moreover, $\dd\omega^0$ vanishes separately on the following subbundles of $\ker\omega^0$, which map isomorphically to the summands of \eqref{Deltasplitdef},
\begin{align*}
&(\pi_0)_*\left(\R\left\{\partial_{\omega^i}\right\}_{i=1}^n\right)=\overrightarrow{\pi_0(\phi)^\bot},
&(\pi_0)_*\left(\R\left\{\partial_{\tilde{\gamma}^0_i}\right\}_{i=1}^n\right)=(\pi_0(\phi)^\bot)^\uparrow;
\end{align*}
name their direct sum $\boldsymbol{\Delta}\subset T\F$ so that $(\pi_0)_*\boldsymbol{\Delta}=\Delta$. We lift the almost-complex structure $J$ on $\Delta\subset TUM$ to $\boldsymbol{\Delta}$ to get $(\pi_0)^*\mathfrak{L}$-Lagrangian subbundles $\boldsymbol{\Lambda},\overline{\boldsymbol{\Lambda}}\subset\C\boldsymbol{\Delta}$ mapping isomorphically under $(\pi_0)_*$ to \eqref{UMLambda}, 
\begin{align*}
&\boldsymbol{\Lambda}=\C\left\{\partial_{\omega^i}-\im\partial_{\tilde{\gamma}^0_i}\right\},
&\overline{\boldsymbol{\Lambda}}=\C\left\{\partial_{\omega^i}+\im\partial_{\tilde{\gamma}^0_i}\right\}.
\end{align*}

The lifted almost-complex structure is encoded in 1-forms $\lambda^\mathtt{i}\in\Omega^1(\F,\C)$ given by
\begin{align}\label{Flambda}
&\lambda^0=\tfrac{1}{2}\omega^0,
&\lambda^i=\tfrac{1}{2}(\omega^i-\im\tilde{\gamma}^i_0).
\end{align}
With this coframing, \eqref{domegai} ($\mathtt{i}=0$) becomes
\begin{equation}\label{Fdlambda0}
\dd\lambda^0=\im\epsilon_{ij}\lambda^i\wedge\overline{\lambda}^j,
\end{equation}
and for $i\geq 1$,
\begin{equation}\label{Fdlambdai}
\begin{aligned}
\dd\lambda^i&=\tfrac{1}{2}\Big(
-\gamma^i_0\wedge\omega^0-\gamma^i_j\wedge\omega^j
+\im(\gamma^i_j\wedge\tilde{\gamma}^j_0-R^i_{0k0}\omega^k\wedge\omega^0-\tfrac{1}{2}R^i_{0kl}\omega^k\wedge\omega^l)\Big)\\
&=-(\gamma^i_0+\im R^i_{0k0}\omega^k)\wedge\lambda^0-\gamma^i_j\wedge\lambda^j
-\tfrac{\im}{4}R^i_{0kl}(\lambda^k+\overline{\lambda}^k)\wedge(\lambda^l+\overline{\lambda}^l).
\end{aligned}
\end{equation}
These are the structure equations \eqref{firstSE} for
\begin{align}\label{UMSEformstorsion}
&\xi^i=\gamma^i_0+\im R^i_{0k0}\omega^k,
&\left.\begin{array}{l}\xi_0\equiv 0\\\xi^i_j\equiv\gamma^i_j\end{array}\right\}\mod\{\im\Re\lambda^k\}
&&N^i_{jk}=-\tfrac{\im}{4}R^i_{0jk}.
\end{align}
\begin{rem}
Continuing the discussion of Remarks \ref{Nijenhuisrem} and \ref{Jintrem}, we now observe that the Nijenhuis tensor of the standard CR structure on $UM$ is given, up to scale, by \eqref{RNijen}. Note that this is trivial if $\dim M=2$ -- the one component of curvature having been absorbed into the definition \eqref{UMSEformstorsion} of $\xi^1$ -- reflecting the fact that 3-dimensional CR manifolds are automatically CR-integrable. To reproduce the theorem of \cite{BarDrag} that Riemannian space forms can give rise to integrable CR structures on $UM$, it is convenient to replace the Riemannian structure equations \eqref{dLeviCiv} with the ``model-mutated" version (see \cite[Ch.5 \S 6]{Sharpe}), 
\begin{align*}
\dd\gamma^\mathtt{i}_\mathtt{j}&=-\gamma^\mathtt{i}_\mathtt{k}\wedge\gamma^\mathtt{k}_\mathtt{j}\pm\omega^{\tt{i}}\wedge\omega^{\tt{j}}+\tfrac{1}{2}R^\mathtt{i}_{\mathtt{jkl}}\omega^\mathtt{k}\wedge\omega^\mathtt{l}, 
\end{align*}
whose curvature tensor measures deviation from the structure equations of space forms with nonzero sectional curvature. The resulting equations \eqref{Fdlambdai}, \eqref{UMSEformstorsion} on $UM$ would see the 1-forms $\xi^i$ replaced by $\gamma^i_0+\im(R^i_{0k0}\omega^k\mp\omega^i)$, and vanishing curvature would ensure CR integrability.
 
\end{rem}

The 1-forms \eqref{Flambda} may be considered a pull-back of the tautological form \eqref{lambdataut} along a section $UM\to\B^1_\Lambda$ of the bundle of coframings which are 1-adapted to \eqref{UMLambdasplit}. We realize the tautological forms \eqref{etatautlambda}, \eqref{Adiag} of $\mathcal{B}^2$ by setting
\begin{align*}
&\eta^0=\lambda^0,
&\eta^i=(1-\im r)\lambda^i-a\overline{\lambda}^i,
\end{align*}
which brings the structure equations \eqref{Fdlambda0}, \eqref{Fdlambdai} into agreement with \eqref{splitLGSE2} for 1-forms
\begin{equation}\label{RicSQSEforms}
\begin{aligned}
\varsigma&=-\im\dd r+\tfrac{1}{2}(a\dd\overline{a}-\overline{a}\dd a)+\tfrac{\im}{2}|z_-|^2\eta^0,\\
\kappa&=-\im a\dd r-(1-\im r)\dd a-\tfrac{\im}{2}(z_-)^2\eta^0-\tfrac{\im}{4}(z_+)^3R_{0k}\overline{\eta}^k,\\
\zeta^i&=\tfrac{\im}{2}|z_-|^2\eta^i-\tfrac{\im}{2}(z_-)^2\overline{\eta}^i-(z_-)\gamma^i_0-\im(z_+) R^i_{0k0}\omega^k,
\end{aligned}
\end{equation}
and torsion coefficients
\begin{align}\label{RicSQSEtorsion}
&O^i_{kl}=-\tfrac{\im}{4}z_+(\overline{z}_+)^2R^i_{0kl},
&P^i_{kl}=-\tfrac{\im}{2}(z_+)^2\overline{z}_+R^i_{0kl},
&&Q^i_{kl}=-\tfrac{\im}{4}(z_+)^3C^i_{0kl},
\end{align}
where we have introduced 
\begin{align}\label{RicWeyl}
&z_\pm=(1-\im r\pm a),
&R_{\mathtt{jk}}=\tfrac{1}{n+1}\epsilon^{\mathtt{m}}_{\mathtt{n}}R^{\mathtt{n}}_{\mathtt{jmk}},
&&C^{\mathtt{i}}_{\mathtt{jkl}}=R^{\mathtt{i}}_{\mathtt{jkl}}-\epsilon^{\mathtt{i}}_{\mathtt{k}}R_{\mathtt{jl}}.
\end{align}

\begin{rem}
The last term in the expression \eqref{RicSQSEforms} of $\kappa$ is exactly the Ricci-shift \eqref{Ricshift}, and distinguishes the torsion coefficients \eqref{RicSQSEtorsion} from those of the standard L-contact structure in that $Q$ is given by components the Weyl curvature tensor \eqref{RicWeyl} rather than the full Riemannian curvature tensor (like $O$ and $P$). One could also absorb the Ricci components of $P$ into $\kappa$, but the analogous effort to use $\varsigma$ to absorb the Ricci components of $O$ and/or $P$ fails to be maintain the identity for $\dd\eta^0$. 
\end{rem}

\begin{thm}\label{mainthm}
In order for the Ricci-shifted L-contact structure of $UM$ to be the leaf space of a 2-nondegenerate CR manifold, it is sufficient that $M$ is conformally flat and real-analytic. 
\end{thm}

\begin{proof}
If $UM$ is the leaf space of a 2-nondegenerate CR structure $\U$, then the Nijenhuis tensor of $\U$ is represented by the torsion coefficients $Q$ in the structure equations \eqref{splitLGSE2}. For the Ricci-shifted L-contact structure equations \eqref{RicSQSEtorsion}, these torsion coefficients are components of the Weyl curvature tensor of $M$. Because the L-contact structure of $UM$ (standard or Ricci-shifted) arises from a split-quaternionic structure, it meets the criteria \cite[Cor. 2.8]{SykesZelenko} for recoverability, provided $M$ is real-analytic. 
\end{proof}

\vspace{\baselineskip}

\subsection{Quaternionic L-contact in Semi-Riemannian Signature $(p+1,p)$}\label{QLUM}

$(M,g)$ is semi-Riemannian, $\mathrm{sig}(g)=(p+1,p)$. With the same notation as \S\ref{ONFBsec}, 
\begin{align}\label{QDelta}
\boldsymbol{\Delta}=\R\left\{\partial_{\omega^i},\partial_{\omega^{p+i}},\partial_{\gamma^0_i},\partial_{\gamma^0_{p+i}}\right\}
\end{align}
has the $\mathfrak{L}$-Lagrangian splitting $\C\boldsymbol{\Delta}=\boldsymbol{\Lambda}\oplus\overline{\boldsymbol{\Lambda}}$ for
\begin{align*}
&\boldsymbol{\Lambda}=\C\left\{\partial_{\omega^i}-\im\partial_{\gamma^0_i}, \partial_{\omega^{p+i}}-\im\partial_{\gamma^0_{p+i}}\right\}. 
\end{align*}
The quaternionic structure $K:\boldsymbol{\Delta}\to\boldsymbol{\Delta}$ maps
\begin{align}\label{QKM}
K:\left\{\begin{array}{c}\partial_{\omega^i}\mapsto\partial_{\omega^{p+i}},\\ \\
\partial_{\gamma^0_i}\mapsto-\partial_{\gamma^0_{p+i}}.\end{array}\right.
\end{align}
The contact form still pulls back to $\F$ according to \eqref{omega0contact}, and the first semi-Riemannian structure equation \eqref{domegai} is
\begin{align}\label{Qdomega0}
\dd\omega^0&=-\gamma^0_j\wedge\omega^j+\gamma^0_{p+j}\wedge\omega^{p+j},
\end{align}
so we define the 1-adapted coframing as before,
\begin{align*}
&\lambda^0=\tfrac{1}{2}\omega^0,
&\lambda^i=\tfrac{1}{2}(\omega^i-\im\gamma^i_0),
&&\lambda^{p+i}=\tfrac{1}{2}(\omega^{p+i}-\im\gamma^{p+i}_0),
\end{align*}
whereby \eqref{Qdomega0} becomes \eqref{Qdlambda0} with $\xi_0=0$. The rest of the construction proceeds by analogy to the end of \S\ref{ONFBsec}, following \S\ref{QSEsec} rather than \S\ref{SQSEsec}. In particular, the definition \eqref{QSEeta} (with \eqref{Adiag}) of the tautological forms on $\B^2$ encodes the induced action of \eqref{QKM} on $\boldsymbol{\Lambda},\overline{\boldsymbol{\Lambda}}$. Details will be included in an updated pre-print.

\vspace{\baselineskip}

\appendix

\section{The Homogeneous Model of \S\ref{SQlinessec} as a Sphere Bundle}\label{appendixsec}

In this appendix, we exhibit a local hypersurface realization of the homogeneous models presented in \S\ref{SQlinessec} in an effort to provide a more analytical perspective on that discussion. Specifically, we offer an explicit local identification of that L-contact structure with the unit tangent bundle of a Riemannian manifold. Three caveats are in order. First, we only present the definite-signature case $\epsilon_{ij}=\delta_{ij}$ \eqref{epsilonmetric}, the mixed signature case being a straightforward modification of this. Second, we use representations of the bilinear and Hermitian forms that differ from \eqref{SQbh} in order to bring the local defining equations in \S\ref{FTsec} into more familiar form. Finally, the local appearance of this model depends on our non-canonical choices of these representations. To better explain this last caveat, we take a detour in \S\ref{Lieconsec} to compare L-contact geometry to a closely related field.

\vspace{\baselineskip}

\subsection{Lie Contact Geometry}\label{Lieconsec}

Strongly regular 2-nondegenerate CR manifolds (\cite[Thms 3.2, 5.1]{PZ}) and L-contact manifolds are generalizations -- in Cartan's sense of \emph{espaces g\'{e}n\'{e}raliz\'{e}s} \cite[Preface]{Sharpe} -- of the homogeneous models presented in \S\S\ref{SQlinessec}-\ref{Qlinessec}: coset spaces of $O(n+2,2)$ (in the definite signature case $\epsilon_{ij}=\delta_{ij}$) for the stabilizer subgroups \eqref{stabsubgps} preserving complex lines and 2-planes in $\C^{n+4}$. In this respect these two geometries are closely related to conformal and Lie contact geometry, respectively. 

Lie contact geometry (\cite[\S4.2.5]{CapSlovak}) is the generalization (in Cartan's sense, but also generalization to arbitrary signature) of Lie Sphere Geometry (\cite{LieSphereBook}), introduced by Sophus himself in his 1872 dissertation. The objects of interest in Lie sphere geometry are oriented $n$-spheres of arbitrary radius in $\R^{n+1}$ -- including limiting cases of points (zero radius) and oriented hyperplanes (infinite radius) -- which can be in ``signed" contact (oriented or not) with each other at some common point of tangency. Symmetries of interest map these objects to each other while preserving their contact relationships, so we are interested in rigid motions $O(n+1)$ of $\R^{n+1}$, Lorentz boosts and time-translations of $\R^{n+1,1}$, conformal symmetries $O(n+2,1)/\{\pm\mathbbm{1}\}$ of the conformal compactification $S^{n+1}$ of $\R^{n+1}$, and everything else generated by these (\cite[Thm 3.16]{LieSphereBook}, cf. \cite[\S2.2]{LaguerreSurf}). All of this is encoded in the $(n+2)$-dimensional Lie quadric $\mathcal{Q}\subset\R\PP^{n+3}$ whose symmetries $O(n+2,2)/\{\pm\mathbbm{1}\}$ preserve the real projectivization of the null cone in $\R^{n+2,2}$. Lie contact geometry generalizes the space of null projective lines on $\mathcal{Q}$, the Grassmannian $\mathsf{Gr}^0_2\R^{n+2,2}$ of totally null, real 2-planes.

The representation of $O(n+2,2)$ pertinent to conformal and Lie contact geometry differs from that of \S\ref{SQlinessec} in that it is totally real; in other words, it is the representation corresponding to the trivial split-quaternionic structure $\K=\mathbbm{1}$ (nevertheless, the significance of split-quaternionic structures to Lie contact geometry has been noticed -- see \cite[\S2.5]{LCquat}). Moreover, the stabilizer subgroups \eqref{stabsubgps} of \emph{real} lines and 2-planes are parabolic in this real representation, hence the Cartan geometries modeled on homogeneous coset spaces of $O(n+2,2)$ with respect to these stabilizer subgroups are parabolic geometries  -- Lorentzian-conformal geometry generalizing $\mathcal{Q}$ and Lie contact generalizing $\mathsf{Gr}^0_2\R^{n+2,2}$, respectively -- unlike 2-nondegenerate CR and L-contact geometries.    

One motivation for the present work was the prominent role of unit tangent bundles as examples of Lie contact manifolds. For a Riemannian manifold $(M,g)$, $UM$ \eqref{UMdef} has a natural Lie contact structure incorporating many of the same constructions as in \S\ref{UTBsec}. The ``flat" (homogeneous) model of Lie contact geometry -- i.e., $\mathsf{Gr}^0_2\R^{n+2,2}$ or the null lines on the Lie quadric $\mathcal{Q}$ -- is the unit tangent bundle of the Riemannian sphere $M=S^{n+1}$. However, it should noted that the Lie contact structure of $UM$ is only associated to the conformal class of the metric $g$ rather than $g$ itself. In other words, flatness of $UM$ as a Lie contact manifold is equivalent to $M$ being \emph{conformally} flat (\cite{LCconformal}), and any such $UM$ is locally equivalent to the homogeneous model. Similar considerations for L-contact manifolds should be kept in mind while reading \S\ref{FTsec}, which presents a local realization of a homogeneous L-contact manifold as a unit tangent bundle.

\vspace{\baselineskip}

\subsection{The Future Tube as a Unit Tangent Bundle}\label{FTsec}

The simplest examples of Levi-degenerate CR hypersurfaces which are not straightenable are tube conical surfaces $\U=C\times\im\R^{m+1}\subset\C^{m+1}$ where $C\subset\R^{m+1}$ is a cone $rC\subset C\ \forall r>0$ (\cite[\S5.2]{ChirkaCR}). The flat model of strongly regular, 2-nondegenerate CR geometry (where the nondegenerate part of the Levi form has definite signature) is the \emph{tube over the future light cone} $\U=\rho^{-1}(0)$ for $\rho:\C^{m+1}\to\R$ given by
\begin{align*}
2\rho=(z_1+\overline{z}_1)^2+\dots+(z_m+\overline{z}_m)^2-(z_{m+1}+\overline{z}_{m+1})^2.
\end{align*}
$\U$ is the tube over the \emph{future} light cone because we impose 
\begin{align}\label{futurecoord}
z_{m+1}+\overline{z}_{m+1}>0.
\end{align}
Note that if we take coordinates $(z_1,\dots,z_{m+1},z_{m+2})\in\C^{m+2}$ to lie in an affine subset of $\C\PP^{m+2}$,
\begin{align*}
[1:z_1:\dots:z_{m+1}:z_{m+2}]=[Z_0:Z_1:\dots:Z_{m+1}:Z_{m+2}],
\end{align*}
then $\{\rho=0\}=\PP(\N_{\boldsymbol{b}}\cap\N_{\boldsymbol{h}})$ where $\N_{\boldsymbol{b}},\N_{\boldsymbol{h}}\subset\C^{m+3}$ are the null cones of the bilinear and Hermitian forms
\begin{align*}
\boldsymbol{b}(Z,Z)&=-2Z_0Z_{m+2}+Z_1^2+\dots+Z_m^2-Z_{m+1}^2,\\
\boldsymbol{h}(Z,Z)&=Z_0\overline{Z}_{m+2}+\overline{Z}_0Z_{m+2}+|Z_1|^2+\dots+|Z_m|^2-|Z_{m+1}|^2.
\end{align*}

The Levi form of $\U$ is
\begin{align*}
\mathcal{L}=\partial\overline{\partial}\rho=\dd z_1\wedge\dd\overline{z}_1+\dots+\dd z_m\wedge\dd\overline{z}_m-\dd z_{m+1}\wedge\dd\overline{z}_{m+1}.
\end{align*}
Evidently, $\partial\rho(X)=0=\mathcal{L}(X,\overline{X})$ for 
\begin{align*}
X=(z_1+\overline{z}_1)\frac{\partial}{\partial z_1}+\dots+(z_m+\overline{z}_m)\frac{\partial}{\partial z_m}+(z_{m+1}+\overline{z}_{m+1})\frac{\partial}{\partial z_{m+1}},
\end{align*}
which is the infinitesimal generator of the scaling operator 
\begin{align}\label{scaleop}
&z\mapsto\delta_\R(z,z)z,
&\delta_\R(z,z)=x_1^2+\dots+x_{m+1}^2,
\end{align}
where we have broken the complex coordinates into their real an imaginary parts
\begin{align*}
z_i=x_i+\im y_i.
\end{align*}

Observe that if $(z_1,\dots,z_{m+1})\in\U$ and $c=r+\im s\in\C$ ($r,s\in\R$, $r>0$), then we also have
\begin{align}\label{conscale}
(cx_1+\im y_1,\dots,cx_{m+1}+\im y_{m+1})=(rx_1+\im(sx_1+y_1),\dots,rx_{m+1}+\im(sx_{m+1}+y_{m+1}))\in\U.
\end{align}
On the real subspace $\R^{m+1}\subset\C^{m+1}$ away from the origin, \eqref{conscale} is just a constant (rescaled) version of \eqref{scaleop}. Thus, the ``complex half rays" on $\U$ given by all such multiples \eqref{conscale} are the leaves of the Levi foliation of $\U$. We claim that the leaf space $U$ of $\U$ is the unit tangent bundle $UM$ for $(M,g)=\R^m$ with its Euclidean metric; i.e., $U=\R^{n+1}\times S^n$ for $n=m-1$.

To see this, let $t=(t_1,\dots,t_m)\in\R^{n+1}$ be coordinates for $M$ and $u=(u_1,\dots,u_m)\in S^n$ coordinates constrained by
\begin{align*}
g(u,u)=u_1^2+\dots+u_m^2=1.
\end{align*}
Define the inclusion $U\to\U$ by
\begin{align}\label{leafinclusion}
(t,u)\mapsto (u_1+\im t_1,\dots,u_m+\im t_m,1).
\end{align}
Note that we implicitly rely on the futuristic condition \eqref{futurecoord}. From here, it is straightforward to verify:
\begin{itemize}
\item distinct points of $UM$ map into distinct leaves in $\U$, 

\item every $(z_1,\dots,z_{m+1})\in\U$ lies in the image \eqref{leafinclusion} composed with \eqref{conscale} for some $c=r+\im s$ with $r>0$. 
\end{itemize} 
$\U$ is therefore realized as a complex ``half-ray" bundle over $U=U\R^{n+1}$ whose fiber over $u\in U$ is $\{cu\ |\ c\in\C,\ \Re c>0\}$.

\bibliographystyle{plain}

\bibliography{References}

\end{document}